\documentclass[12pt,reqno]{amsart}
\usepackage[margin=1in]{geometry}
\usepackage{amsmath,amssymb,amsthm,graphicx,amsxtra, setspace}
\usepackage[utf8]{inputenc}
\usepackage{mathrsfs}
\usepackage{hyperref}
\usepackage{upgreek}
\usepackage{mathtools}
\usepackage[mathcal]{euscript}
\usepackage{xcolor}
\allowdisplaybreaks

\usepackage[pagewise]{lineno}

\usepackage{graphicx,eurosym}
\usepackage{hyperref}
\usepackage{mathtools}

\usepackage[cyr]{aeguill}

\colorlet{darkblue}{blue!50!black}

\hypersetup{
	colorlinks,%
	citecolor=blue,%
	filecolor=red,%
	linkcolor=darkblue,%
	urlcolor=blue,%
	pdfnewwindow=true,%
	pdfstartview={FitH}
}

\usepackage{graphicx,amscd,mathrsfs,wrapfig,mathrsfs,lipsum}
\usepackage{eufrak}
\usepackage{tikz}
\usepackage{multicol}
\usepackage{caption}
\usetikzlibrary{arrows}

\colorlet{darkblue}{blue!50!black}

\makeatletter

\renewcommand{\tocsection}[3]{%
	\indentlabel{\@ifnotempty{#2}{\bfseries\ignorespaces#1 #2\quad}}\bfseries#3}
\renewcommand{\tocsubsection}[3]{%
	\indentlabel{\@ifnotempty{#2}{\ignorespaces#1 #2\quad}}#3}
\let\oldtocsubsubsection=\tocsubsubsection
\renewcommand{\tocsubsubsection}[2]{\hspace{2em}\oldtocsubsubsection{#1}{#2}}

\newcommand\@dotsep{4.5}
\def\@tocline#1#2#3#4#5#6#7{\relax
	\ifnum #1>\c@tocdepth 
	\else
	\par \addpenalty\@secpenalty\addvspace{#2}%
	\begingroup \hyphenpenalty\@M
	\@ifempty{#4}{%
		\@tempdima\csname r@tocindent\number#1\endcsname\relax
	}{%
		\@tempdima#4\relax
	}%
	\parindent\z@ \leftskip#3\relax \advance\leftskip\@tempdima\relax
	\rightskip\@pnumwidth plus1em \parfillskip-\@pnumwidth
	#5\leavevmode\hskip-\@tempdima{#6}\nobreak
	\leaders\hbox{$\m@th\mkern \@dotsep mu\hbox{.}\mkern \@dotsep mu$}\hfill
	\nobreak
	\hbox to\@pnumwidth{\@tocpagenum{\ifnum#1=1\bfseries\fi#7}}\par
	\nobreak
	\endgroup
	\fi}
\AtBeginDocument{%
	\expandafter\renewcommand\csname r@tocindent0\endcsname{0pt}
}
\def\l@subsection{\@tocline{2}{0pt}{2.5pc}{5pc}{}}
\makeatother

\newtheorem{theorem}{Theorem}[section]
\newtheorem{lemma}[theorem]{Lemma}

\newtheorem{definition}[theorem]{Definition}
\newtheorem{example}[theorem]{Example}
\newtheorem{remark}[theorem]{Remark}

\let\originalleft\left
\let\originalright\right
\renewcommand{\left}{\mathopen{}\mathclose\bgroup\originalleft}
\renewcommand{\right}{\aftergroup\egroup\originalright}

\let\oldtocsection=\tocsection

\let\oldtocsubsection=\tocsubsection

\let\oldtocsubsubsection=\tocsubsubsection

\renewcommand{\tocsection}[2]{\hspace{0em}\oldtocsection{#1}{#2}}
\renewcommand{\tocsubsection}[2]{\hspace{1em}\oldtocsubsection{#1}{#2}}
\renewcommand{\tocsubsubsection}[2]{\hspace{2em}\oldtocsubsubsection{#1}{#2}}

\renewcommand{\d}{\/\mathrm{d}\/}

\def\L{\mathbb{L}}

\def\C{\mathrm{C}}
\def\f{\boldsymbol{f}}

\def\D{\mathrm{D}}

\def\X{\mathbb{X}}

\def\g{\boldsymbol{g}}

\def\V{\mathbb{v}}

\def\g{\boldsymbol{g}}
\def\G{\boldsymbol{G}}

\def\no{\nonumber}
\def\V{\mathbb{V}}
\def\wi{\widetilde}

\def\u{\mathrm{U}}

\def\u{\boldsymbol{u}}
\def\H{\mathbb{H}}
\def\n{\boldsymbol{n}}

\newcommand{\R}{\mathbb{R}}

\renewcommand{\d}{\/\mathrm{d}\/}

\newcommand{\Addresses}{{
		\footnote{
			
			\noindent \textsuperscript{1,2}Department of Mathematics, Indian Institute of Technology Roorkee-IIT Roorkee,
			Haridwar Highway, Roorkee, Uttarakhand 247667, INDIA.\par\nopagebreak
			\noindent 	\textit{e-mail:} \texttt{Pardeep Kumar: pkumar3@ma.iitr.ac.in.}
			
			\noindent  \textit{e-mail:} \texttt{Manil T. Mohan: manilfma@iitr.ac.in, maniltmohan@gmail.com.}
			
			\noindent \textsuperscript{*}Corresponding author.

			\textit{Key words:} Convective Brinkman-Forchheimer equations, inverse source problem, integral overdetermination condition, contraction mapping theorem, well-posedness. 
			
			Mathematics Subject Classification (2020): Primary 35R30; Secondary 35Q35, 35Q30.

}}}

\begin{document}
	
	
	\title[An inverse problem for 2D CBF equations]{Existence, uniqueness and stability of an inverse problem for two-dimensional convective Brinkman-Forchheimer equations with the integral overdetermination
		\Addresses}
	\author[P. Kumar and M. T. Mohan ]{Pardeep Kumar\textsuperscript{1} and Manil T. Mohan\textsuperscript{2*}}

	\maketitle
	
	\begin{abstract}
		In this article, we study an inverse problem for the following convective Brinkman-Forchheimer (CBF) equations:
		\begin{align*}
			\boldsymbol{u}_t-\mu \Delta\boldsymbol{u}+(\boldsymbol{u}\cdot\nabla)\boldsymbol{u}+\alpha\boldsymbol{u}+\beta|\boldsymbol{u}|^{r-1}\boldsymbol{u}+\nabla p=\boldsymbol{F}:=f \g, \ \ \  \nabla\cdot\boldsymbol{u}=0,
		\end{align*}
		in a bounded domain $\Omega\subset\mathbb{R}^2$ with smooth boundary $\partial\Omega$, where $\alpha,\beta,\mu>0$ and $r\in[1,3]$. The investigated inverse problem consists of reconstructing the vector-valued velocity function $\boldsymbol{u}$, the pressure field $p$ and the scalar function $f$. For the divergence free  initial data $\u_0 \in \L^2(\Omega)$, we prove the existence of a solution to the inverse problem for  two-dimensional CBF equations with the integral overdetermination condition, by showing the existence of a unique fixed point for an equivalent operator equation (using an extension of the contraction mapping theorem). Moreover, we establish the uniqueness and Lipschitz stability results of the solution to the inverse problem for 2D CBF equations with $r \in[1,3]$.
	\end{abstract}

\tableofcontents

	\section{Introduction}\label{sec1}\setcounter{equation}{0}
The convective Brinkman-Forchheimer (CBF) equations describe the motion of incompressible fluid flows in a saturated porous medium (cf. \cite{AB}). This article's primary goal is to discuss the well-posedness of an inverse problem to CBF equations in two dimensions for a divergence free initial data in $\L^2(\Omega)$.
	\subsection{The mathematical model and the direct problem}
	 Let  $\Omega\subset\R^2$  be a bounded domain with a smooth boundary $\partial\Omega$ (at least $\C^2$-boundary). The CBF equations are given by
	\begin{align}
		\u_t-\mu \Delta\u+(\u\cdot\nabla)\u+\alpha\u+\beta|\u|^{r-1}\u+\nabla p=\boldsymbol{F}&:=f \g, \ \text{ in } \ \Omega\times(0,T), \label{1a}\\ \nabla\cdot\u&=0, \ \ \ \ \text{ in } \ \Omega\times(0,T),\label{1b}
	\end{align}
	with initial condition
	\begin{align}\label{1c}
		\u=\u_0, \ \text{ in } \ \Omega \times \{0\},
	\end{align}
	and	boundary condition
	\begin{align}\label{1d}
		\u=\boldsymbol{0},\ \ \text{ on } \ \partial\Omega\times[0,T).
	\end{align}
	Here $\u(x,t) \in \R^2$ represents the velocity field at position $x$ and time $t$, $p(x,t)\in\R$ denotes the pressure field and $\boldsymbol{F}(x,t)\in\R^2$ stands for the external forcing. The constant $\mu$ denotes the positive Brinkman coefficient (effective viscosity), the positive constants $\alpha$ and $\beta$ stand for the Darcy coefficient (permeability of porous medium) and the Forchheimer coefficient (proportional to the porosity of the material), respectively (cf. \cite{KT2}). The absorption exponent $r\in[1,\infty)$ and the cases, $r=3$ and $r>3$, are known as the critical exponent and the fast growing nonlinearity, respectively.  For $\alpha=\beta=0$, we obtain the classical $2 \D$ Navier-Stokes equations (NSE). Thus, one can consider the equations \eqref{1a}-\eqref{1d} as a modification (by introducing an absorption term $\alpha\u+\beta|\u|^{r-1}\u$) of the classical NSE. Thus, one may consider the model \eqref{1a}-\eqref{1d} as NSE with damping.  In order to obtain the uniqueness of the pressure $p$, one can impose the condition $\int_{\Omega}p(x,t)\d x=0, $ for $t\in [0,T]$.   The model given in \eqref{1a}-\eqref{1d} is recognized to be more accurate when the flow velocity is too large for the Darcy's law to be valid alone, and apart from that, the porosity is not too small, so that we call these types of models as \emph{non-Darcy models} (cf. \cite{PAM}).   	  It has been proved  in Proposition 1.1, \cite{KWH}  that the critical homogeneous CBF equations have the same scaling as NSE only when $\alpha=0$ and no scale invariance property for other values of $\alpha$ and $r$. 
	
	Let us now discuss some results available in the literature on the global solvability of the system \eqref{1a}-\eqref{1d} (direct problem). The existence and	uniqueness of weak as well as strong solutions in two and three dimensional bounded domains are available in \cite{SNA,CLF,KT2,MTM4}, etc.,  for the  global solvability results in periodic domains and whole space,	one can refer to \cite{ZCQJ,KWH,ZZXW}, etc. The Navier-Stokes problem modified an  absorption term $|\u|^{r-1}\u$, for $r>1$,  in bounded domains with compact boundary is considered in \cite{SNA}. The existence of Leray-Hopf weak solutions,  for any dimension $d\geq 2$, and its uniqueness for $d=2$ is established in \cite{SNA}. For $d=2,r\in[1,\infty)$ and $d=3,r\in[3,\infty)$ ($2\beta\mu\geq 1$ for $r=3$), the existence and uniqueness of global Leray-Hopf  weak as well as strong solution is established in \cite{MTM4}. 
	\subsection{The inverse problem}
	 Even though the direct problem is important, but it requires the knowledge of physical parameters such as the Brinkman coefficient $\mu$, Darcy coefficient $\alpha$, Forchheimer coefficient $\beta$ and the forcing term $\boldsymbol{F}:=f \g$. When in addition to the solution of the equation, recovery of some physical properties of the investigated object or the effects of the external sources are needed, we require inverse problems to determine a coefficient or to handle the right hand side of the differential equation arising in mathematical modeling of physical phenomena. In such modeled problems, it is more efficient to consider the inverse source problems. However, posing an inverse problem requires some additional information of the solution besides the given initial and boundary conditions. An additional information on a solution of the inverse problem can be  integral overdetermination condition, which is the case considered in this paper.	The integral overdetermination condition considered in this work is given by 
	\begin{align}\label{1e}
		\int_{\Omega}\u(x,t) \cdot \boldsymbol{\omega}(x) \d x=\varphi(t), \ \ t \in [0,T],
	\end{align}
	where $\varphi(t)$ is the measurement data, which  stands for the average velocity on the domain $\Omega$ and	$\boldsymbol{\omega}$ is given quantity, which  corresponds to the type of device used to measure the velocity.
	
	Let the vector-valued external forcing $\boldsymbol{F}$ appearing in \eqref{1a} be represented by $$\boldsymbol{F}(x,t):=f(t) \g(x,t),$$ where the scalar function $f$ is unknown and $\g$ is a given vector-valued function.	The investigation of the nonlinear inverse problem in this article consists of reconstructing the vector-valued velocity function $\u$, the pressure field $p$ and the scalar function $f$ from the system \eqref{1a}-\eqref{1d}, with the integral overdetermination condition \eqref{1e}, and the given functions $\u_0, \boldsymbol{\omega}, \varphi$ and $\g$.
	
The	inverse problems with   final overdetermination  (cf. \cite{IB,NLG,VI1,VI,SY}, etc.) and  integral overdetermination (cf.  \cite{SNAKK,JFGN,YJJF,KK,PKKK,Pk,POV,SGP} etc.) conditions have been well studied in the literature. The existence and uniqueness of the generalized solution of the inverse problem for the nonlinear nonstationary system of Navier-Stokes equations with integral overdetermination is investigated in \cite{IAV}. The existence results of an inverse problem to  NSE with both the integral as well as  final overdetermination data are  proved in  \cite{POV} by using Schauder's fixed point theorem in two and three dimensions. The global well-posedness of the inverse source problem for parabolic systems is examined in \cite{PT}. Under the assumption that the  initial data $\u_0 \in \H$ and the viscosity constant is sufficiently large, the authors in \cite{JF} proved that an inverse problem for 2D NSE with the final overdetermination data is well-posed. To prove the same, they have used Tikhonov's fixed point theorem. In \cite{JFMD}, the authors considered an inverse problem of determining a viscosity coefficient in NSE by observing data in a neighborhood of the boundary. 	For an extensive study on various inverse problems corresponding to NSE and related models, where one requires  to reconstruct the density of external forces or some coefficients of the equations on the basis of final or functional  overdetermination, we refer the interested readers to  \cite{MBFC,AYC2,AYC,AYC1,MC,OYMY,VI,AIK,RYL,POV,PV,VP}, etc., and references therein. 

The existence and uniqueness of  an inverse problem for three-dimensional	nonlinear equations of Kelvin-Voigt fluids is proved in \cite{KK}.  Using the contraction mapping principle, a local in time existence and uniqueness result of an inverse problem for Kelvin-Voigt fluid flow equations with memory term and integral overdetermination condition is obtained in \cite{PKKK}. The authors  in \cite{SNAKK} established the local unique solvability result of an inverse problem for generalized Kelvin-Voigt equation with $p$-Laplacian and damping term with integral overdetermination condition. Recently, the well-posedness of an inverse problem for 2D and 3D convective Brinkman-Forchheimer equations with  final overdetermination data using Schauder's fixed point theorem is examined in \cite{PKMTM}. The results in \cite{PKMTM} are obtained by  assuming sufficient smoothness on the given data ($\u_0\in\H^2(\Omega)\cap\V$). 
\subsection{Technical difficulties}
 We emphasize that, due to a technical difficulty in working with bounded domains, the method employed in \cite{JF,POV,IAV}, etc. (for the initial data $\u_0\in\H$) may only be applicable for $\ r\in[1,3]$ in this article. Note that in the case of bounded domains,  $\mathrm{P}_{\H}(|\u|^{r-1}\u)$ ($\mathrm{P}_{\H}$ is the Helmholtz-Hodge orthogonal projection, see subsection \ref{sub2.1}, \cite{PKMTM}) need not be zero on the boundary, and $\mathrm{P}_{\H}$ and $-\Delta$ are not necessarily commuting (for a counter example, see Example 2.19, \cite{JCR4}). Furthermore, while taking the inner product with $-\Delta\u$ in \eqref{1a}, $-\Delta\u\cdot\n |_{\partial \Omega}\neq 0$ in general and the term with pressure will not vanish (see \cite{KT2}). As a result, the equality (\cite{KWH})
\begin{align*}
	&\int_{\Omega}(-\Delta\u(x))\cdot|\u(x)|^{r-1}\u(x)\d x\no\\&
	=\int_{\Omega} |\nabla\u(x)|^{r-1}\u(x) \d x + 4\bigg( \frac{r-1}{(r+1)^2}\bigg) \int_{\Omega}| \nabla|\u(x)|^\frac{r+1}{2}|^r \d x 
	\nonumber\\&=\int_{\Omega}|\nabla\u(x)|^2|\u(x)|^{r-1}\d x+\frac{r-1}{4}\int_{\Omega}|\u(x)|^{r-3}|\nabla|\u(x)|^2|^2\d x,
\end{align*}
may not be helpful in bounded domains.

	\subsection{Main results}
	By a solution of the inverse problem  \eqref{1a}-\eqref{1e}, we mean by a triplet $(\u, p,f)$ such that 
	\begin{align*}
		&\u\in\mathrm{L}^{\infty}(0,T;\H)\cap\mathrm{L}^{2}(0,T;\V) \cap \mathrm{L}^{r+1}(0,T;\widetilde{\L}^{r+1}),  \ p\in \mathrm{L}^\frac{r+1}{r}(0,T;\mathrm{L}_0^2(\Omega)), \ f \in \mathrm{L}^2(0,T),
	\end{align*}
and in addition, all the relations \eqref{1a}-\eqref{1e} hold. In order to prove the existence and uniqueness of solutions for the above investigated inverse problem, we use the method developed in \cite{POV,IAV}. Note that the work \cite{POV,IAV} does not take into account of the stability of the solution. The major goals of this article are
	\begin{enumerate}
		\item [(i)] the existence and uniqueness of the solution, 
		\item [(ii)] the stability of the solution in the norm of corresponding function spaces, 
	\end{enumerate}
	to the inverse problem \eqref{1a}-\eqref{1e} under the assumptions $\u_0 \in \H, \ \g \in \C([0,T];\ \L^2(\Omega)), \ \varphi \in \mathrm{H}^1(0,T)$, using an extension of the contraction mapping theorem. In contrast to the results obtained for CBF equations in \cite{PKMTM}, the well-posedness of the generalized solution of the inverse problem holds for the initial data $\u_0 \in \H$ in two dimensions for $r \in [1,3]$ and the additional information given in the form of integral overdetermination condition.
	\begin{definition}[\cite{IAV}]\label{def}
		Let $r \in[1,3]$. A triplet $(\u,p,f)$ is said to be a \emph{weak solution} of a nonlinear inverse problem \eqref{1a}-\eqref{1e}  if
		$$\u \in \mathrm{L}^{\infty}(0,T;\H) \cap \mathrm{L}^2(0,T;\V) \cap \mathrm{L}^{r+1}(0,T;\widetilde{\L}^{r+1}), \ p\in \mathrm{L}^\frac{r+1}{r}(0,T;\mathrm{L}_0^2(\Omega)),$$
		$$\int_{\Omega}\u(x,t) \cdot \mathbf{v}(x) \d x \in \mathrm{H}^1(0,T), \  \text{ for all }\  \mathbf{v} \in \V; \ \ f \in \mathrm{L}^2(0,T),$$
		and they satisfy the integral identity
		\begin{align}\label{1f}
			&\frac{\d}{\d t}	\int_{\Omega}\u(x,t) \cdot \mathbf{v}(x) \d x+\mu 	\int_{\Omega}\nabla\u(x,t) \cdot \nabla \mathbf{v}(x) \d x\nonumber\\&+	\int_{\Omega}(\u(x,t) \cdot \nabla)\u(x,t) \cdot \mathbf{v}(x) \d x+\alpha 	\int_{\Omega}\u(x,t) \cdot \mathbf{v}(x) \d x\nonumber\\&+\beta 	\int_{\Omega}|\u(x,t)|^{r-1}\u(x,t) \cdot \mathbf{v}(x) \d x=\int_{\Omega}f(t) \g(x,t) \cdot \mathbf{v}(x) \d x,
		\end{align}
		  along with the initial condition \eqref{1c} and the integral overdetermination condition \eqref{1e}, for any $\mathbf{v} \in \V$.
	\end{definition}
	In this definition, we have incorporated the incompressibility of the fluid velocity and the boundary condition \eqref{1d} as in Definition \ref{def} in the sense that the function $\u(\cdot,t) \in \V$ for a.e. $t \in [0,T]$.	Let us now state the main result in this work on the well-posedness of solutions of the inverse problem \eqref{1a}-\eqref{1e}.
	\begin{theorem}\label{thm2}
		Let $\Omega \subset \mathbb{R}^2, \ \u_0 \in \H, \ \g \in \C([0,T];\L^2(\Omega)), \ \boldsymbol{\omega} \in \H^2(\Omega) \cap \V,  \ \varphi \in \mathrm{H}^1(0,T)$  $$ \text{ and } \ \left| \int_\Omega\g(x,t) \cdot \boldsymbol{\omega}(x) \d x\right| \geq g_0 >0, \ \ (g_0 \equiv \text{constant}), \ \ \ 0\leq t \leq T,$$ and in addition, $\nabla\boldsymbol{\omega}\in\wi\L^{\infty}$ and the compatibility condition
		\begin{align}\label{1i}
			\int_{\Omega}\u_0(x) \cdot \boldsymbol{\omega}(x) \d x=\varphi(0),
		\end{align}
	be satisfied. Then under the assumptions of Theorem \ref{thm3.2} (see below), the following assertions are satisfied for the inverse problem \eqref{1a}-\eqref{1e}:
		\begin{enumerate}
			\item [(i)]  There exists a unique weak solution $(\u,p,f)$ to the inverse problem \eqref{1a}-\eqref{1e}.
			\item [(ii)]  Let $(\u_i,p_i,f_i)$ $(i=1,2)$ be two solutions of the inverse problem \eqref{1a}-\eqref{1e} corresponding to the input data  $(\u_{0i},\varphi_i,\g_i) \ (i=1,2)$. Then there exists a constant $C$ such that
			\begin{align}\label{1k}
				&\nonumber\|\u_1-\u_2\|_{\mathrm{L}^\infty(0,T;\H)}+	\|\u_1-\u_2\|_{\mathrm{L}^2(0,T;\V)}+	\|\u_1-\u_2\|_{\mathrm{L}^{r+1}(0,T;\widetilde{\L}^{r+1})}\\&\quad+\|p_1-p_2\|_{\mathrm{L}^\frac{r+1}{r}(0,T;\mathrm{L}^2(\Omega))}+\|f_1-f_2\|_{\mathrm{L}^2(0,T)} \nonumber\\&\leq C\bigg(\|\u_{01}-\u_{02}\|_{ \H}+\sup_{t\in[0,T]}\|(\g_1-\g_2)(t)\|_{\L^2}+\|\varphi_1-\varphi_2\|_{\mathrm{H}^1(0,T)}\bigg),
			\end{align}
			where $C$ depends on the input data, $\mu,\alpha,\beta,r$,$T$ and $\Omega$.
		\end{enumerate}	
	\end{theorem}
	\noindent
The rest of the paper is organized as follows: In section \ref{sec2}, we first state and prove the relation between the solvability of the inverse problem \eqref{1a}-\eqref{1e} and an equivalent operator equation of the second kind (Theorem \ref{thm1}). In section \ref{sec3}, we prove our main result, that is, Theorem \ref{thm2}, by first showing the existence of a solution to the equivalent operator equation by using the extension of the contraction mapping theorem and then demonstrating the uniqueness and stability of the solution to the inverse problem. In Appendix \ref{sec4}, we deduce a number of a-priori estimates that are needed to handle the inverse problem \eqref{1a}-\eqref{1e}.
	\section{Mathematical formulation}\label{sec2}\setcounter{equation}{0}
In this section, we state and prove Theorem \ref{thm1}, which converts the inverse problem \eqref{1a}-\eqref{1e} into an equivalent nonlinear operator equation of the second kind \eqref{1h} and prove their equivalence. We commence this section by introducing function spaces and some standard notations  (cf. Section 2.1, \cite{MTM4,MTM6}), which will be used throughout the paper. 
	\subsection{Function spaces}\label{sub2.1}  Let $\Omega \subset \R^2$ be a bounded domain  with smooth boundary $\partial \Omega$. Let $\C_0^{\infty}(\Omega;\R^2)$ be the space of all infinitely differentiable functions  ($\R^2$-valued) with compact support in $\Omega\subset\R^2$.  Let us define 
	\begin{align*} 
		\mathcal{V}&:=\{\u\in\C_0^{\infty}(\Omega,\R^2):\nabla\cdot\u=0\},\\
		\mathbb{H}&:=\text{the closure of }\ \mathcal{V} \ \text{ in the Lebesgue space } \L^2(\Omega)=\mathrm{L}^2(\Omega;\R^2),\\
		\mathbb{V}&:=\text{the closure of }\ \mathcal{V} \ \text{ in the Sobolev space } \H^1(\Omega)=\mathrm{H}^1(\Omega;\R^2),\\
		\widetilde{\L}^{p}&:=\text{the closure of }\ \mathcal{V} \ \text{ in the Lebesgue space } \L^p(\Omega)=\mathrm{L}^p(\Omega;\R^2),
	\end{align*}
	for $p\in(2,\infty]$. Then, under some smoothness assumptions on the boundary, we characterize the spaces $\H$, $\V$, $\widetilde{\L}^p$ and $\widetilde{\L}^\infty$ as 
	$
	\H=\{\u\in\L^2(\Omega):\nabla\cdot\u=0,\u\cdot\boldsymbol{n}\big|_{\partial\Omega}=0\},$ where $\boldsymbol{n}$ is the unit outward drawn normal to $\partial\Omega$, and $\u\cdot\n\big|_{\partial\Omega}$ should be understood in the sense of trace in $\H^{-1/2}(\partial\Omega)$ (cf. Theorem 1.2, Chapter 1, \cite{Te}),   with the norm  $\|\u\|_{\H}^2:=\int_{\Omega}|\u(x)|^2\d x,
	$
	$
	\V=\{\u\in\H_0^1(\Omega):\nabla\cdot\u=0\},$  with the norm $ \|\u\|_{\V}^2:=\int_{\Omega}|\nabla\u(x)|^2\d x $ (since $\Omega$ is a bounded smooth domain),  $\widetilde{\L}^p=\{\u\in\L^p(\Omega):\nabla\cdot\u=0, \u\cdot\boldsymbol{n}\big|_{\partial\Omega}=0\},$ with the norm $\|\u\|_{\widetilde{\L}^p}^p:=\int_{\Omega}|\u(x)|^p\d x$, and $\widetilde{\L}^\infty=\{\u\in\L^\infty(\Omega):\nabla\cdot\u=0, \u\cdot\boldsymbol{n}\big|_{\partial\Omega}=0\},$ with the norm $\|\u\|_{\widetilde{\L}^\infty}:=\operatorname*{ess\,sup}\limits_{x \in \Omega}|\u(x)|$, respectively. Let $(\cdot,\cdot)$ denote the inner product in the Hilbert space $\H$ and $\langle\cdot,\cdot\rangle$ represents the duality pairing between the spaces $\V$ and its dual $\V'$, and $\wi\L^p$ and its dual $\wi\L^{\frac{p}{p-1}}$. The following well-known Ladyzhenkaya inequality (cf. Lemma 1, \cite{OAL}) $$\|\u\|_{\L^4}\leq 2^{1/4}\|\u\|_{\L^2}^{1/2}\|\nabla\u\|_{\L^2}^{1/2},\ \text{ for all }\ \u\in\H_0^1(\Omega)$$ is used repeatedly in the paper. 
	Wherever needed, we assume that $p_0\in\mathrm{L}^2_0(\Omega),$ where $\mathrm{L}^2_0(\Omega):=\left\{p\in\mathrm{L}^2(\Omega):\int_{\Omega}p(x)\d x=0\right\}$. In the sequel, $C$ denotes a generic constant which may take different values at different places.
	
	Later discussions of the equivalent formulation of an inverse problem \eqref{1a}-\eqref{1e} are based on the works \cite{POV,IAV}, in which the authors established a  relationship between the solvability of the inverse problem for NSE in two dimensions and an operator equation of the second kind.
	\subsection{Equivalent formulation}
	We require the following  input data function requirements:
	$$\u_0 \in \H, \ \g \in \C([0,T];\L^2(\Omega)), \ \boldsymbol{\omega} \in \H^2(\Omega) \cap \V, \  \nabla\boldsymbol{\omega}\in\wi\L^{\infty}, \ \varphi \in \mathrm{H}^1(0,T),$$ $$\text{and} \ \ \left| \int_\Omega\g(x,t) \cdot \boldsymbol{\omega}(x) \d x\right| \geq g_0 >0, \ \ (g_0 \equiv \text{constant}), \ \ \ 0\leq t \leq T.$$
	Under the above assumptions, we derive an operator equation	of the second kind for the scalar function $f$. We define the nonlinear operator $\mathcal{A}:\mathrm{L}^2(0,T) \to \mathrm{L}^2(0,T)$ by
	\begin{align}\label{1g}
		(\mathcal{A}f)(t):=	\frac{1}{g_1(t)}\bigg\{
		\int_{\Omega} \big(\mu\nabla\u \cdot \nabla \boldsymbol{\omega} +	(\u \cdot \nabla)\u \cdot \boldsymbol{\omega}+\alpha 	\u \cdot \boldsymbol{\omega} +\beta 	|\u|^{r-1}\u \cdot \boldsymbol{\omega}\big) \d x+\varphi '(t)\bigg\},
	\end{align}
	where $\u$ is obtained by using the weak solution of the direct problem \eqref{1a}-\eqref{1d} and 
	\begin{align}\label{1g1}
		g_1(t)=\int_{\Omega} \g(x,t)\cdot \boldsymbol{\omega}(x) \d x.
	\end{align}
	We analyze the following nonlinear operator equation of the second kind over the space $\mathrm{L}^2(0,T)$:
	\begin{align}\label{1h}
		f=\mathcal{A}f.
	\end{align}
The solvability of the inverse problem	\eqref{1a}-\eqref{1e} is inevitably connected to the fixed points of the operator $\mathcal{A}$. For this purpose, we first consider a closed ball $\mathcal{D}$ in 
space $\mathrm{L}^2(0,T)$ with center $\varphi'/g_1$ such that
\begin{align}\label{3a1}
	\mathcal{D}=\big\{f: f \in \mathrm{L}^2(0,T) \text{ and } \|f-\varphi'/g_1\|_{\mathrm{L}^2(0,T)} \leq a \big\},
\end{align}
where $$g_1(t)=\int_\Omega\g(x,t) \cdot \boldsymbol{\omega}(x) \d x.$$ 
	\begin{remark}
	Since the equation \eqref{1a} is taken in the divergence free	space $\V$,	the pressure $p(\cdot,\cdot)$ does not appear in the equations \eqref{1f} and \eqref{1g}. The pressure $p(\cdot,\cdot)$ can be recovered from the equation \eqref{1a} after reconstructing the pair $(\u ,f)$.
\end{remark}
	The relation between the solvability of the inverse problem \eqref{1a}-\eqref{1e} and the nonlinear operator equation of the second kind \eqref{1h} is verified by the following theorem.
	\begin{theorem}\label{thm1}
		Let $\Omega \subset \mathbb{R}^2, \ \u_0 \in \H, \ \g \in \C([0,T];\L^2(\Omega)), \ \boldsymbol{\omega} \in \H^2(\Omega) \cap \V,  \ \varphi \in \mathrm{H}^1(0,T)$  $$ \text{ and } \ \left| \int_\Omega\g(x,t) \cdot \boldsymbol{\omega}(x) \d x\right| \geq g_0 >0, \ \ (g_0 \equiv \text{constant}), \ \ \ 0\leq t \leq T.$$
		Then the following assertions hold true:\
		\begin{enumerate}
			\item 	[(i)] If the inverse problem \eqref{1a}-\eqref{1e} has a solution, then the operator equation \eqref{1h} also has a solution.
			\item 	[(ii)] If the operator equation \eqref{1h} has a solution and the compatibility condition \eqref{1i}
			is satisfied, then there exists a solution of the inverse problem \eqref{1a}-\eqref{1e}.
		\end{enumerate}
	\end{theorem}
	\begin{proof}
		Let us first prove (i). Let the inverse problem \eqref{1a}-\eqref{1e} has a solution, say $(\u,p,f)$. Multiplying both sides of \eqref{1a} by $\boldsymbol{\omega}(x)$, and integrating by parts, we obtain 
		\begin{align}\label{2a}
			&\frac{\d}{\d t}\int_{\Omega}\u(x,t) \cdot \boldsymbol{\omega}(x) \d x+\mu 	\int_{\Omega}\nabla\u(x,t) \cdot \nabla\boldsymbol{\omega}(x) \d x \nonumber\\&\quad+	\int_{\Omega}(\u(x,t) \cdot \nabla)\u(x,t) \cdot \boldsymbol{\omega}(x) \d x+\alpha 	\int_{\Omega}\u(x,t) \cdot \boldsymbol{\omega}(x) \d x\nonumber\\&\quad+\beta 	\int_{\Omega}|\u(x,t)|^{r-1}\u(x,t) \cdot \boldsymbol{\omega}(x) \d x= f(t)	\int_{\Omega}\g(x,t) \cdot \boldsymbol{\omega}(x) \d x.
		\end{align}
		From the definition \eqref{1g} of the operator $\mathcal{A}$, the integral overdetermination condition \eqref{1e}, and assumption \eqref{1g1}, one can easily deduce from \eqref{2a} that
		$$ \mathcal{A}f=f.$$ This implies that the function $f$ solves the operator equation \eqref{1h}; thereby proving (i).\
		
		We proceed to prove (ii). Let the operator equation \eqref{1h} has a solution, say $f \in \mathrm{L^2}(0,T)$. Upon substituting $f$ into \eqref{1a}, we use the system \eqref{1a}-\eqref{1d} to find a unique weak solution $(\u(\cdot),p(\cdot))$ of the direct problem \eqref{1a}-\eqref{1d}. We claim that the function $\u(\cdot)$ satisfies the integral overdetermination condition \eqref{1e}. By inserting  $\mathbf{v}(x)=\boldsymbol{\omega}(x)$ in \eqref{1f}, and using assumption \eqref{1g1}, we arrive at 
		\begin{align}\label{2b}
			&\frac{\d}{\d t}\int_{\Omega}\u(x,t) \cdot \boldsymbol{\omega}(x) \d x+\mu 	\int_{\Omega}\nabla\u(x,t) \cdot \nabla\boldsymbol{\omega}(x) \d x\nonumber\\&\quad+	\int_{\Omega}(\u(x,t) \cdot \nabla)\u(x,t) \cdot \boldsymbol{\omega}(x) \d x+\alpha 	\int_{\Omega}\u(x,t) \cdot \boldsymbol{\omega}(x) \d x \nonumber\\&\quad+\beta 	\int_{\Omega}|\u(x,t)|^{r-1}\u(x,t) \cdot \boldsymbol{\omega}(x) \d x=f(t) g_1(t).
		\end{align}
		On the other hand, since the function $f$ is a solution of the operator equation \eqref{1h}, we have
		\begin{align}\label{2c}
			&\varphi'(t)+\mu 	\int_{\Omega}\nabla\u(x,t) \cdot \nabla\boldsymbol{\omega}(x) \d x+	\int_{\Omega}(\u(x,t) \cdot \nabla)\u(x,t) \cdot \boldsymbol{\omega}(x) \d x \nonumber\\&\quad+\alpha 	\int_{\Omega}\u(x,t) \cdot \boldsymbol{\omega}(x) \d x+\beta 	\int_{\Omega}|\u(x,t)|^{r-1}\u(x,t) \cdot \boldsymbol{\omega}(x) \d x=f(t) g_1(t).
		\end{align}
		Subtracting the equation \eqref{2c} from the equation \eqref{2b}, we find 
		\begin{align}\label{2d}
			\frac{\d}{\d t}\int_{\Omega}\u(x,t) \cdot \boldsymbol{\omega}(x) \d x-\varphi'(t)=0.
		\end{align}
		Integrating the equation \eqref{2d} with respect to time $t$ from $0$ to $t$, and then using the compatibility condition \eqref{1i}, we obtain
		\begin{align*}
			\int_{\Omega}\u(x,t) \cdot \boldsymbol{\omega}(x) \d x=\varphi(t),
		\end{align*}
		which is the integral overdetermination condition \eqref{1e}. Having obtained this, we can say that the triplet $(\u,p,f)$ is a weak solution of the inverse problem \eqref{1a}-\eqref{1e} and this proves the second assertion of the theorem.
	\end{proof}

	\section{Proof of Theorem \ref{thm2}}\label{sec3}\setcounter{equation}{0}
	The energy estimates obtained in Appendix \ref{sec4} allow us to demonstrate the existence and uniqueness of  solution to the inverse problem \eqref{1a}-\eqref{1e} as well as the stability of the solution obtained. For proving the existence of a solution to the  inverse problem \eqref{1a}-\eqref{1e}, by using Theorem \ref{thm1}, it suffices to show that the nonlinear operator $\mathcal{A}$ has a fixed point in $\mathcal{D}$ and this follows from an application of the extension of the contraction mapping theorem.	Subsequently, arguments for the existence and uniqueness of the solution of the inverse problem \eqref{1a}-\eqref{1e} are based on the works \cite{POV,IAV}, where the existence  and uniqueness of a solution of the inverse problem for 2D NSE has been investigated, by exploiting an extension of the contraction mapping theorem.
	\begin{theorem}[Theorem 2.4, \cite{SLA}]\label{thm3.1}
		Let $(\X,d)$ be a complete metric space and let $\mathcal{A} : \X \to \X$ be a	mapping such that for some positive integer $k$, $\mathcal{A}^k$ is a contraction on $\X$. Then, $\mathcal{A}$	has a unique fixed point.
	\end{theorem}
	\subsection{Existence} The solvability of the inverse problem  \eqref{1a}-\eqref{1e} is inevitably connected to the fixed points of the operator $\mathcal{A}$. Now, we demonstrate that the operator $\mathcal A$ has a unique fixed point.	The following theorem provides sufficient conditions under which the operator $\mathcal{A}$ maps the closed ball $\mathcal{D}$ into itself.
	\begin{theorem}\label{thm3.2}
		Let $\Omega \subset \mathbb{R}^2, \ \u_0 \in \H, \ \g \in \C([0,T];\L^2(\Omega)), \ \boldsymbol{\omega} \in \H^2(\Omega) \cap \V, \ \nabla \boldsymbol{\omega} \in \widetilde{\L}^{\infty}, \ \varphi \in \mathrm{H}^1(0,T)$ and $$\left| \int_\Omega\g(x,t) \cdot \boldsymbol{\omega}(x) \d x\right| \geq g_0 >0, \ \ (g_0\equiv \text{constant}), \ \ \ 0\leq t \leq T.$$
		\begin{enumerate}
			\item [(i)]  For $r \in (2,3]$, if 
			\begin{align}\label{3a}
				m_1 <a,
			\end{align}
			where 
			\begin{align}\label{3.3}
				m_1&=\frac{C}{g_0} \bigg[T\bigg\{
				\big(\|\Delta \boldsymbol{\omega}\|_{\H}+\|\boldsymbol{\omega}\|_{\H}\big)\bigg(\|\u_0\|_{\H}+T^{1/2}\widetilde{a}\sup_{t\in[0,T]}\|\g(t)\|_{\L^2}\bigg)\nonumber\\&\qquad+\|\nabla \boldsymbol{\omega}\|_{\widetilde{\L}^{\infty}}\bigg(\|\u_0\|_{\H}^2+ \widetilde{a}^2\sup_{t\in[0,T]}\|\g(t)\|_{\L^2}^2\bigg)\bigg\}^2\nonumber\\&\qquad+  T^\frac{3-r}{r-1}\|\boldsymbol{\omega}\|_{\H^2}^2\bigg(\|\u_0\|_{\H}+T^{1/2}\widetilde{a}\sup_{t\in[0,T]}\|\g(t)\|_{\L^2}\bigg)^\frac{4}{(r-1)}
				\no\\&\qquad \times \bigg(\|\u_0\|_{\H}^2+\widetilde{a}^2\sup_{t\in[0,T]}\|\g(t)\|_{\L^2}^2\bigg)^\frac{2(r-2)}{(r-1)}\bigg]^\frac{1}{2},
			\end{align}
			$$\widetilde{a}=a+\|\varphi'/g_1\|_{\mathrm{L}^2(0,T)},$$ and $a$ is the radius of the closed ball $\mathcal{D}$.
			\item [(ii)] For $r \in [1,2]$, if 
			\begin{align}\label{3b}
				m_2 <a,
			\end{align}
			where 
			\begin{align}\label{3.5}
				m_2&= \frac{CT^{1/2}}{g_0} \bigg\{
				\big(\|\Delta \boldsymbol{\omega}\|_{\H}+\|\boldsymbol{\omega}\|_{\H}\big)\bigg(\|\u_0\|_{\H}+T^{1/2}\widetilde{a}\sup_{t\in[0,T]}\|\g(t)\|_{\L^2}\bigg)\no\\&\quad+\|\nabla \boldsymbol{\omega}\|_{\widetilde{\L}^{\infty}}\bigg(\|\u_0\|_{\H}^2+ \widetilde{a}^2\sup_{t\in[0,T]}\|\g(t)\|_{\L^2}^2\bigg)
				\no\\&\quad +\boldsymbol{\omega}\|_{\H^2}\bigg(\|\u_0\|_{\H}+T^{1/2}\widetilde{a}\sup_{t\in[0,T]}\|\g(t)\|_{\L^2}\bigg)^r\bigg\}.
			\end{align}
		\end{enumerate}
	 Then the operator $\mathcal{A}$ maps the closed ball $\mathcal{D}$ into itself.
	\end{theorem}
	\begin{proof}
		Let $f$ be an arbitrary fixed function in $\mathcal{D}$, then by the definition of the ball $\mathcal{D}$, we deduce that 
		\begin{align}\label{3c}
			\|f\|_{\mathrm{L}^2(0,T)} \leq \widetilde{a}.
		\end{align}
		\vskip 0.2 cm
		\noindent\textbf{Case I:} \emph{ $r\in (2,3]$.}
		The norm of the function $\mathcal{A}f-\varphi'/g_1$ can be estimated as follows:
		\begin{align}\label{3d}
			\|\mathcal{A}f-\varphi'/g_1\|_{\mathrm{L}^2(0,T)}&=\bigg(\int_0^T|\mathcal{A}f-\varphi'/g_1|^2 \d t\bigg)^\frac{1}{2}\nonumber\\&\leq \frac{1}{g_0} \bigg[\int_0^T\bigg(\mu \|\Delta \boldsymbol{\omega}\|_{\H}\|\u(t)\|_{\H}+\|\nabla \boldsymbol{\omega}\|_{\widetilde{\L}^{\infty}}\|\u(t)\|_{\H}^2\nonumber\\&\quad+\alpha \|\boldsymbol{\omega}\|_{\H}\|\u(t)\|_{\H}+\beta \|\boldsymbol{\omega}\|_{\widetilde{\L}^\infty} \|\u(t)\|_{\widetilde{\L}^{r}}^{r}\bigg)^2 \d t\bigg]^\frac{1}{2}.
		\end{align}
		We estimate $\int_0^T\|\u(t)\|_{\widetilde{\L}^{r}}^{2r} \d t$ by using interpolation and H\"older's inequalities as
		\begin{align}\label{3e}
			\int_0^T	\|\u(t)\|_{\widetilde{\L}^{r}}^{2r} \d t &\leq \int_0^T\left(	\|\u(t)\|_{\H}^\frac{2}{r(r-1)}	\|\u(t)\|_{\widetilde{\L}^{r+1}}^\frac{(r-2)(r+1)}{r(r-1)} \right)^{2r} \d t\no\\& \leq \int_0^T \left(	\|\u(t)\|_{\H}^\frac{4}{(r-1)}	\|\u(t)\|_{\widetilde{\L}^{r+1}}^\frac{2(r-2)(r+1)}{(r-1)} \right) \d t \no\\& \leq T^\frac{3-r}{r-1} \sup_{t\in[0,T]}\|\u(t)\|_{\H}^\frac{4}{(r-1)} \left(\int_0^T \|\u(t)\|_{\widetilde{\L}^{r+1}}^{r+1}\d t\right)^\frac{2(r-2)}{(r-1)}.
		\end{align}
		Substituting the estimate \eqref{3e} in \eqref{3d}, we obtain
		\begin{align}\label{3f}
			&\|\mathcal{A}f-\varphi'/g_1\|_{\mathrm{L}^2(0,T)}
			\no\\&\quad\leq \frac{C}{g_0} \bigg[T\bigg\{ \big(\|\Delta \boldsymbol{\omega}\|_{\H}+\|\boldsymbol{\omega}\|_{\H}\big)\sup_{t\in[0,T]}\|\u(t)\|_{\H}+\|\nabla \boldsymbol{\omega}\|_{\widetilde{\L}^{\infty}}\sup_{t\in[0,T]}\|\u(t)\|_{\H}^2\bigg\}^2\no\\&\qquad+T^\frac{3-r}{r-1} \|\boldsymbol{\omega}\|_{\H^2}^2 \sup_{t\in[0,T]} \|\u(t)\|_{\H}^\frac{4}{(r-1)} \left(\int_0^T \|\u(t)\|_{\widetilde{\L}^{r+1}}^{r+1}\d t\right)^\frac{2(r-2)}{(r-1)}\bigg]^\frac{1}{2}.
		\end{align}
		With reference to \eqref{E1} and \eqref{E2}, it can be seen from \eqref{3f} that
		\begin{align}\label{3f1}
			&\|\mathcal{A}f-\varphi'/g_1\|_{\mathrm{L}^2(0,T)}	\nonumber\\&\quad\leq \frac{C}{g_0} \bigg[T\bigg\{
			\big(\|\Delta \boldsymbol{\omega}\|_{\H}+\|\boldsymbol{\omega}\|_{\H}\big)\bigg(\|\u_0\|_{\H}+T^{1/2}\sup_{t\in[0,T]}\|\g(t)\|_{\L^2}\|f\|_{\mathrm{L^2}(0,T)}\bigg)\nonumber\\&\qquad+\|\nabla \boldsymbol{\omega}\|_{\widetilde{\L}^{\infty}}\bigg(\|\u_0\|_{\H}^2+\frac{1}{\alpha}\sup_{t\in[0,T]}\|\g(t)\|_{\L^2}^2\|f\|_{\mathrm{L^2}(0,T)}^2\bigg)\bigg\}^2\nonumber\\&\qquad+  T^\frac{3-r}{r-1}\|\boldsymbol{\omega}\|_{\H^2}^2\bigg(\|\u_0\|_{\H}+T^{1/2}\sup_{t\in[0,T]}\|\g(t)\|_{\L^2}\|f\|_{\mathrm{L^2}(0,T)}\bigg)^\frac{4}{(r-1)}
			\no\\&\qquad \times \bigg\{\frac{1}{2\beta}\bigg(\|\u_0\|_{\H}^2+\frac{1}{\alpha}\sup_{t\in[0,T]}\|\g(t)\|_{\L^2}^2\|f\|_{\mathrm{L^2}(0,T)}^2\bigg)\bigg\}^\frac{2(r-2)}{(r-1)}\bigg]^\frac{1}{2}.
		\end{align}
		Using \eqref{3c} in \eqref{3f1}, we deduce that
		\begin{align}\label{3f2}
			\|\mathcal{A}f-\varphi'/g_1\|_{\mathrm{L}^2(0,T)} \leq m_1,
		\end{align}
		where $m_1$ is defined in \eqref{3.3}. One can take $T$ sufficiently small such that $m_1 <a$ (if needed $a$ can be chosen sufficiently large so that $a>\|\u_0\|_{\H}$).  For the case $r=3$, apart from smallness of $T$, one may need to restrict the input data also. Using this relation, the estimate \eqref{3f2} immediately implies that the nonlinear operator $\mathcal{A}$ brings the ball $\mathcal{D}$ into itself.
		
		\vskip 0.2 cm
		\noindent\textbf{Case II:} \emph{ $r\in [1,2]$.}  From \eqref{3d},	the norm of the function $\mathcal{A}f-\varphi'/g_1$ can be estimated as follows  
		\begin{align}\label{3g}
			&\|\mathcal{A}f-\varphi'/g_1\|_{\mathrm{L}^2(0,T)}
			\no\\&\quad\leq \frac{C}{g_0} \bigg[T\bigg\{ \big(\|\Delta \boldsymbol{\omega}\|_{\H}+\|\boldsymbol{\omega}\|_{\H}\big)\sup_{t\in[0,T]}\|\u(t)\|_{\H}+\|\nabla \boldsymbol{\omega}\|_{\widetilde{\L}^{\infty}}\sup_{t\in[0,T]}\|\u(t)\|_{\H}^2\bigg\}^2\no\\&\qquad+|\Omega|^{2-r} \|\boldsymbol{\omega}\|_{\H^2}^2\int_0^T\|\u(t)\|_{\H}^{2r} \d t \bigg]^\frac{1}{2}
			\no\\&\quad\leq \frac{CT^{1/2}}{g_0} \bigg\{ \big(\|\Delta \boldsymbol{\omega}\|_{\H}+\|\boldsymbol{\omega}\|_{\H}\big)\sup_{t\in[0,T]}\|\u(t)\|_{\H}+\|\nabla \boldsymbol{\omega}\|_{\widetilde{\L}^{\infty}}\sup_{t\in[0,T]}\|\u(t)\|_{\H}^2\no\\&\qquad+ \|\boldsymbol{\omega}\|_{\H^2}\sup_{t\in[0,T]}\|\u(t)\|_{\H}^{r} \bigg\},
		\end{align}
		where $|\Omega|$ is the Lebesgue measure of $\Omega$.	Substituting the estimates \eqref{E1} and \eqref{E2} in the inequality \eqref{3g}, we get 
		\begin{align}\label{3g1}
			&\|\mathcal{A}f-\varphi'/g_1\|_{\mathrm{L}^2(0,T)}
			\no\\&\quad\leq \frac{CT^{1/2}}{g_0}  \bigg\{
			\big(\|\Delta \boldsymbol{\omega}\|_{\H}+\|\boldsymbol{\omega}\|_{\H}\big)\bigg(\|\u_0\|_{\H}+T^{1/2}\sup_{t\in[0,T]}\|\g(t)\|_{\L^2}\|f\|_{\mathrm{L^2}(0,T)}\bigg)\nonumber\\&\qquad+\|\nabla \boldsymbol{\omega}\|_{\widetilde{\L}^{\infty}}\bigg(\|\u_0\|_{\H}^2+\frac{1}{\alpha}\sup_{t\in[0,T]}\|\g(t)\|_{\L^2}^2\|f\|_{\mathrm{L^2}(0,T)}^2\bigg)\no\\&\qquad
			+\|\boldsymbol{\omega}\|_{\H^2}\bigg(\|\u_0\|_{\H}+T^{1/2}\sup_{t\in[0,T]}\|\g(t)\|_{\L^2}\|f\|_{\mathrm{L^2}(0,T)}\bigg)^r\bigg\}.
		\end{align}
		Using the relation \eqref{3c} in \eqref{3g1} results to
		\begin{align}\label{3h}
			\|\mathcal{A}f-\varphi'/g_1\|_{\mathrm{L}^2(0,T)} \leq m_2,
		\end{align}
		where $m_2$ is defined in \eqref{3.5}. One can choose $T$ sufficiently small so that $m_2 <a$. Considering this relationship, the estimate \eqref{3h} implies that the nonlinear operator $\mathcal{A}$ brings the ball $\mathcal{D}$ into itself.\
		
		As a result, for $r \in [1,3]$, the  nonlinear operator $\mathcal{A}$ maps the ball $\mathcal{D}$ into itself, completing the proof of the theorem.
	\end{proof}
	The next theorem establishes the existence of a solution to the inverse problem \eqref{1a}-\eqref{1e}.
	\begin{theorem}
		Let $\Omega \subset \mathbb{R}^2, \ \u_0 \in \H, \ \g \in \C([0,T];\L^2(\Omega)), \ \boldsymbol{\omega} \in \H^2(\Omega) \cap \V, \ \nabla \boldsymbol{\omega} \in \widetilde{\L}^{\infty}, \ \varphi \in \mathrm{H}^1(0,T)$ and $$\left| \int_\Omega\g(x,t) \cdot \boldsymbol{\omega}(x) \d x\right| \geq g_0 >0, \ \ (g_0\equiv constant), \ \ \ 0\leq t \leq T.$$
		Also assume that the nonlinear operator $\mathcal{A}$ maps the ball $\mathcal{D}$ into itself. Then there exists a positive integer $k$ such that the operator $\mathcal{A}^k$ is a contraction mapping in the ball $\mathcal{D}$.
	\end{theorem} 
	\begin{proof}
		Let both of the functions $f_1$ and $f_2$ belong to the ball $\mathcal{D}$. Let $\u_1$ and $\u_2$ be the functions corresponding to the coefficients $f_1$ and
		$f_2$, respectively. By the definition of the operator $\mathcal{A}$, we have 
		\begin{align}\label{3i}
			\left\|\mathcal{A}f_1-\mathcal{A}f_2\right\|_{\mathrm{L}^2(0,T)}^2&=\int_0^T\left|\mathcal{A}f_1-\mathcal{A}f_2\right|^2 \d t\no\\&\leq \frac{1}{g_0^2} \int_0^T\bigg(\mu \|(\u_1-\u_2)(t)\|_{\H}\|\Delta \boldsymbol{\omega}\|_{\H}+\|\nabla \boldsymbol{\omega}\|_{\widetilde{\L}^{\infty}}\|(\u_1-\u_2)(t)\|_{\H}\no\\&\quad \times\big(\|\u_1(t)\|_{\H}+\|\u_2(t)\|_{\H}\big)+\alpha \|\boldsymbol{\omega}\|_{\H}\|(\u_1-\u_2)(t)\|_{\H}\no\\&\quad+\beta \|\boldsymbol{\omega}\|_{\widetilde{\L}^\infty} \||\u_1(t)|^{r-1}\u_1(t)-|\u_2(t)|^{r-1}\u_2(t)\|_{\L^1}\bigg)^2 \d t
			\no\\&\leq \frac{C}{g_0^2} \bigg[T\sup_{t\in[0,T]}\|(\u_1-\u_2)(t)\|_{\H}^2
			\bigg\{ \|\Delta \boldsymbol{\omega}\|_{\H}+\|\boldsymbol{\omega}\|_{\H}\no\\&\quad+\|\nabla \boldsymbol{\omega}\|_{\widetilde{\L}^{\infty}}\bigg(\sup_{t\in[0,T]}\|\u_1(t)\|_{\H}+\sup_{t\in[0,T]}\|\u_2(t)\|_{\H}\bigg)\bigg\}^2\no\\&\quad+ \|\boldsymbol{\omega}\|_{\H^2}^2 \int_0^T\left\||\u_1(t)|^{r-1}\u_1(t)-|\u_2(t)|^{r-1}\u_2(t)\right\|_{\L^1}^2 \d t\bigg].
		\end{align}
		Let us define $h(\u):=|\u|^{r-1}\u$. Then, by applying Taylor's formula (Theorem 7.9.1, \cite{PGC}), we obtain (cf. \cite{MTM6})
		\begin{align}\label{3j}
			&\int_0^T\left\||\u_1|^{r-1}\u_1-|\u_2|^{r-1}\u_2\right\|_{\L^1}^2 \d t\no\\&=\int_0^T\bigg\|\int_0^1 h'(\theta \u_1+(1-\theta)\u_2) \d \theta (\u_1-\u_2)\bigg\|_{\L^1}^2 \d t
			\nonumber\\&\leq \int_0^T\sup_{0<\theta<1}\big\|\big(\u_1(t)-\u_2(t)\big)\big|\theta\u_1(t)+(1-\theta)\u_2(t)\big|^{r-1} \no\\&\quad+(r-1)\big(\theta\u_1(t)+(1-\theta)\u_2(t)\big)\big|\theta\u_1(t)+(1-\theta)\u_2(t)\big|^{r-3} \no\\& \quad\times\big((\theta\u_1(t)+(1-\theta)\u_2(t)) \cdot (\u_1(t)-\u_2(t))\big) \big\|_{\L^1}^2 \d t
			 \no\\& \leq C\int_0^T \sup_{0 <\theta<1}\big\||\theta \u_1(t)+(1-\theta)\u_2(t)|^{r-1}\big\|_{\H}^2\|(\u_1-\u_2)(t)\|_{\H}^2 \d t
			\no\\&	\leq C\int_0^T \|(\u_1-\u_2)(t)\|_{\H}^2 \left( \| \u_1(t)\|_{\widetilde{\L}^{2(r-1)}}^{2(r-1)}+ \| \u_2(t)\|_{\widetilde{\L}^{2(r-1)}}^{2(r-1)}\right)\d t
			\no\\&\leq C \sup_{t\in[0,T]}\|(\u_1-\u_2)(t)\|_{\H}^2\int_0^T\left( \| \u_1(t)\|_{\widetilde{\L}^{2(r-1)}}^{2(r-1)}+ \| \u_2(t)\|_{\widetilde{\L}^{2(r-1)}}^{2(r-1)}\right)\d t.
		\end{align}
		\vskip 0.2 cm
		\noindent\textbf{Case I:} \emph{ $r\in (2,3]$.}
		An application of  Gagliardo-Nirenberg's and H\"older's inequalities in \eqref{3j} yields
		\begin{align}\label{3j1}
			&\int_0^T\left\||\u_1(t)|^{r-1}\u_1(t)-|\u_2(t)|^{r-1}\u_2(t)\right\|_{\L^1}^2 \d t
			\no\\&\quad\leq C \sup_{t\in[0,T]}\|(\u_1-\u_2)(t)\|_{\H}^2\int_0^T\left(\|\u_1(t)\|_{\H}^2 \| \u_1(t)\|_{\V}^{2(r-2)}+ \|\u_2(t)\|_{\H}^2 \| \u_2(t)\|_{\V}^{2(r-2)}\right)\d t
			\no\\&\quad\leq C \sup_{t\in[0,T]}\|(\u_1-\u_2)(t)\|_{\H}^2
			\bigg(\sup_{t\in[0,T]} \| \u_1(t)\|_{\H}^{2}\int_0^T \| \u_1(t)\|_{\V}^{2(r-2)} \d t	\no\\&\qquad+\sup_{t\in[0,T]}\| \u_2(t)\|_{\H}^{2}\int_0^T \| \u_2(t)\|_{\V}^{2(r-2)} \d t\bigg)
			\no\\&\quad\leq C T^{3-r} \sup_{t\in[0,T]}\|(\u_1-\u_2)(t)\|_{\H}^2 \bigg\{\sup_{t\in[0,T]} \| \u_1(t)\|_{\H}^{2}\bigg(\int_0^T \| \u_1(t)\|_{\V}^2 \d t\bigg)^{r-2}\no\\&\qquad+\sup_{t\in[0,T]}\| \u_2(t)\|_{\H}^{2}\bigg(\int_0^T \| \u_2(t)\|_{\V}^2 \d t\bigg)^{r-2}\bigg\}.
		\end{align}
		Substituting the relations \eqref{E61} and \eqref{3j1} in \eqref{3i}, we arrive at 
		\begin{align}\label{3k}
			&\left\|\mathcal{A}f_1-\mathcal{A}f_2\right\|_{\mathrm{L}^2(0,T)}^2
			\no\\&\quad\leq \frac{C}{g_0^2} \bigg(\sup_{t\in[0,T]}\|\g(t)\|_{\L^2}^2\|f_1-f_2\|_{\mathrm{L}^2(0,T)}^2\bigg) \exp \left(\frac{2}{\mu}\int_0^T\|\u_2(t)\|_{\V}^2\d t \right)\no\\&\qquad\times\bigg[T \bigg\{ \|\Delta \boldsymbol{\omega}\|_{\H}+\|\boldsymbol{\omega}\|_{\H}+\|\nabla \boldsymbol{\omega}\|_{\widetilde{\L}^{\infty}}\bigg(\sup_{t\in[0,T]}\|\u_1(t)\|_{\H}+\sup_{t\in[0,T]}\|\u_2(t)\|_{\H}\bigg)\bigg\}^2\no\\&\qquad+ T^{3-r}\|\boldsymbol{\omega}\|_{\H^2}^2 \bigg\{\sup_{t\in[0,T]} \| \u_1(t)\|_{\H}^{2}\bigg(\int_0^T \| \u_1(t)\|_{\V}^2 \d t\bigg)^{r-2}\no\\&\qquad+\sup_{t\in[0,T]}\| \u_2(t)\|_{\H}^{2}\bigg(\int_0^T \| \u_2(t)\|_{\V}^2 \d t\bigg)^{r-2}\bigg\}\bigg]. 
		\end{align}
		The relation \eqref{3k} simplifies to
		\begin{align}\label{3l}
			\left\|\mathcal{A}f_1-\mathcal{A}f_2\right\|_{\mathrm{L}^2(0,T)}^2 \leq m_3 \|f_1-f_2\|_{\mathrm{L}^2(0,T)}^2,
		\end{align}
		where
		\begin{align*}
			m_3&= \frac{CT^{3-r}}{g_0^2} \sup_{t\in[0,T]}\|\g(t)\|_{\L^2}^2 \exp \left(\int_0^T\|\u_2(t)\|_{\V}^2\d t \right)
			\no\\&\quad\times\bigg[T^{r-2} \bigg\{ \|\Delta \boldsymbol{\omega}\|_{\H}+\|\boldsymbol{\omega}\|_{\H}+\|\nabla \boldsymbol{\omega}\|_{\widetilde{\L}^{\infty}}\bigg(\sup_{t\in[0,T]}\|\u_1(t)\|_{\H}+\sup_{t\in[0,T]}\|\u_2(t)\|_{\H}\bigg)\bigg\}^2\no\\&\quad+ \|\boldsymbol{\omega}\|_{\H^2}^2 \bigg\{\sup_{t\in[0,T]} \| \u_1(t)\|_{\H}^{2}\bigg(\int_0^T \| \u_1(t)\|_{\V}^2 \d t\bigg)^{r-2}\no\\&\quad+\sup_{t\in[0,T]}\| \u_2(t)\|_{\H}^{2}\bigg(\int_0^T \| \u_2(t)\|_{\V}^2 \d t\bigg)^{r-2}\bigg\}\bigg].
		\end{align*}
		Using the energy estimates \eqref{E1} and \eqref{E2}, the quantity $m_3$ can be bounded by
		\begin{align}\label{3m}
			m_3 &\leq \frac{CT^{3-r}}{g_0^2}\sup_{t\in[0,T]}\|\g(t)\|_{\L^2}^2\exp \bigg\{\|\u_0\|_{\H}^2+\frac{1}{\alpha}\sup_{t\in[0,T]}\|\g(t)\|_{\L^2}^2\|f_2\|_{\mathrm{L}^2(0,T)}^2\bigg\} \no\\&\quad \times\bigg[T^{r-2}\bigg\{\|\nabla \boldsymbol{\omega}\|_{\widetilde{\L}^{\infty}} \bigg(2 \|\u_0\|_{\H}+T^{1/2}\sup_{t\in[0,T]}\|\g(t)\|_{\L^2}\no\\&\quad \times\left(\|f_1\|_{\mathrm{L}^2(0,T)}+\|f_2\|_{\mathrm{L}^2(0,T)}\right)\bigg)+ \|\boldsymbol{\omega}\|_{\H} + \|\Delta \boldsymbol{\omega}\|_{\H}\bigg\}^2\no\\&\quad+ \|\boldsymbol{\omega}\|_{\H^2}^2\bigg\{2 \|\u_0\|_{\H}^2+\frac{1}{\alpha}\sup_{t\in[0,T]}\|\g(t)\|_{\L^2}^2\left(\|f_1\|_{\mathrm{L}^2(0,T)}^2+\|f_2\|_{\mathrm{L}^2(0,T)}^2\right)\bigg\}^{r-1}\bigg]   .
		\end{align}
		Since the functions $f_1$ and $f_2$ lie within the ball $\mathcal{D}$, a combination of the inequalities \eqref{3l} and \eqref{3m} gives the estimate
		\begin{align}\label{3n}
			\left\|\mathcal{A}f_1-\mathcal{A}f_2\right\|_{\mathrm{L}^2(0,T)} \leq \left(m_4 T^{3-r} \right)^{1/2} \|f_1-f_2\|_{\mathrm{L}^2(0,T)},
		\end{align}
		where
		\begin{align*}
			m_4&= \frac{C}{g_0^2}\sup_{t\in[0,T]}\|\g(t)\|_{\L^2}^2\exp \bigg\{\|\u_0\|_{\H}^2+\widetilde{a}^2\sup_{t\in[0,T]}\|\g(t)\|_{\L^2}^2\bigg\}\\&\quad\times\bigg[T^{r-2}\bigg\{ \|\Delta \boldsymbol{\omega}\|_{\H}+ \|\boldsymbol{\omega}\|_{\H} +\|\nabla \boldsymbol{\omega}\|_{\widetilde{\L}^{\infty}} \bigg(\|\u_0\|_{\H}+T^{1/2}\widetilde{a}\sup_{t\in[0,T]}\|\g(t)\|_{\L^2}\bigg)\bigg\}^2\\&\quad+ \|\boldsymbol{\omega}\|_{\H^2}^2\bigg( \|\u_0\|_{\H}^2+\widetilde{a}^2\sup_{t\in[0,T]}\|\g(t)\|_{\L^2}^2\bigg)^{r-1}\bigg] \ \text{ and }\\
			\widetilde{a}&=a+\|\varphi'/g_1\|_{\mathrm{L}^2(0,T)},
		\end{align*}
		and $a$ is the radius of the ball $\mathcal{D}$.
		Note that $m_4$ is expressed only in terms of the input data.
		
		In accordance with the assumption made, the operator $\mathcal{A}$ maps the ball $\mathcal{D}$ into itself which makes it possible to define for any positive integer $k$, the $k^{\mathrm{th}}$ degree of the operator $\mathcal{A}$. Let this operator be denoted by the symbol $\mathcal{A}^k$. Using the standard technique of mathematical induction on $k$ the inequality \eqref{3n} validates the estimate
		\begin{align}\label{3o}
			\left\|\mathcal{A}^kf_1-\mathcal{A}^kf_2\right\|_{\mathrm{L}^2(0,T)} \leq \bigg(\frac{m_4^k T^{(3-r)k}}{k!}\bigg)^{1/2} \|f_1-f_2\|_{\mathrm{L}^2(0,T)}.
		\end{align}
		It is visible that
		\begin{align*}
			\big(m_4^k T^{(3-r)k}/k!\big)^{1/2} \rightarrow 0  \ \ \text{ as } \ \ k \rightarrow \infty.
		\end{align*}
		This implies that there exists a positive integer $k_0$ such that
		\begin{align*}
			\big(m_4^{k_0}T^{(3-r)k_{0}} /k_{0}!\big)^{1/2} <1.
		\end{align*}
		Due to estimate \eqref{3o}, one can conclude that the operator $\mathcal{A}^{k_0}$ is a contracting mapping in the ball $\mathcal{D}$. 
		\vskip 0.2 cm
		\noindent\textbf{Case II:} \emph{ $r\in [1,2]$.} Applying H\"older's inequality in \eqref{3j} gives
		\begin{align}\label{3o1}
			&\int_0^T\left\||\u_1|^{r-1}\u_1-|\u_2|^{r-1}\u_2\right\|_{\L^1}^2 \d t
			\no\\&\quad\leq C \sup_{t\in[0,T]}\|(\u_1-\u_2)(t)\|_{\H}^2\int_0^T\left(|\Omega|^{2-r} \| \u_1(t)\|_{\H}^{2(r-1)}+ |\Omega|^{2-r} \| \u_2(t)\|_{\H}^{2(r-1)}\right)\d t
			\no\\&\quad\leq C T \sup_{t\in[0,T]}\|(\u_1-\u_2)(t)\|_{\H}^2 \bigg(\sup_{t\in[0,T]}\| \u_1(t)\|_{\H}^{2(r-1)}+\sup_{t\in[0,T]}\| \u_2(t)\|_{\H}^{2(r-1)}\bigg),
		\end{align}
		where $|\Omega|$ is the Lebesgue measure of $\Omega$. Substitution of  the relations \eqref{E61} and \eqref{3o1} in \eqref{3i} results to 
		\begin{align*}
			&\left\|\mathcal{A}f_1-\mathcal{A}f_2\right\|_{\mathrm{L}^2(0,T)}^2
			\\&\leq \frac{CT}{g_0^2} \bigg(\sup_{t\in[0,T]}\|\g(t)\|_{\L^2}^2\|f_1-f_2\|_{\mathrm{L}^2(0,T)}^2\bigg) \exp \left(\frac{2}{\mu}\int_0^T\|\u_2(t)\|_{\V}^2\d t \right)\\&\quad \times \bigg\{ \|\Delta \boldsymbol{\omega}\|_{\H}+\|\boldsymbol{\omega}\|_{\H}+\|\nabla \boldsymbol{\omega}\|_{\widetilde{\L}^{\infty}}\bigg(\sup_{t\in[0,T]}\|\u_1(t)\|_{\H}+\sup_{t\in[0,T]}\|\u_2(t)\|_{\H}\bigg) \\&\quad+ \|\boldsymbol{\omega}\|_{\H^2} \bigg(\sup_{t\in[0,T]}\| \u_1(t)\|_{\H}+\sup_{t\in[0,T]}\| \u_2(t)\|_{\H}\bigg)^{r-1}\bigg\}^2.
		\end{align*}
		From the above estimate, it follows that
		\begin{align}\label{3p}
			\left\|\mathcal{A}f_1-\mathcal{A}f_2\right\|_{\mathrm{L}^2(0,T)}^2 \leq m_5 \|f_1-f_2\|_{\mathrm{L}^2(0,T)}^2,
		\end{align}
		where
		\begin{align*}
			m_5&= \frac{CT}{g_0^2} \sup_{t\in[0,T]}\|\g(t)\|_{\L^2}^2 \exp \left(\int_0^T\|\u_2(t)\|_{\V}^2\d t \right)\\&\quad\times \bigg\{ \|\Delta \boldsymbol{\omega}\|_{\H}+\|\boldsymbol{\omega}\|_{\H}+\|\nabla \boldsymbol{\omega}\|_{\widetilde{\L}^{\infty}}\bigg(\sup_{t\in[0,T]}\|\u_1(t)\|_{\H}+\sup_{t\in[0,T]}\|\u_2(t)\|_{\H}\bigg)	
			\\&\quad+\|\boldsymbol{\omega}\|_{\H^2} \bigg(\sup_{t\in[0,T]}\| \u_1(t)\|_{\H}+\sup_{t\in[0,T]}\| \u_2(t)\|_{\H}\bigg)^{r-1}\bigg\}^2. 
		\end{align*}
		Using the estimates \eqref{E1} and \eqref{E2}, the quantity $m_5$ can be bounded as
		\begin{align}\label{3q1}
			m_5 &\leq \frac{CT}{g_0^2} \sup_{t\in[0,T]}\|\g(t)\|_{\L^2}^2\exp \bigg\{\|\u_0\|_{\H}^2+\frac{1}{\alpha}\sup_{t\in[0,T]}\|\g(t)\|_{\L^2}^2\|f_2\|_{\mathrm{L}^2(0,T)}^2\bigg\}\bigg[ \|\Delta \boldsymbol{\omega}\|_{\H}+ \|\boldsymbol{\omega}\|_{\H} \no\\&\quad+\|\nabla \boldsymbol{\omega}\|_{\widetilde{\L}^{\infty}} \bigg(2 \|\u_0\|_{\H}+T^{1/2}\sup_{t\in[0,T]}\|\g(t)\|_{\L^2}\big(\|f_1\|_{\mathrm{L}^2(0,T)}+\|f_2\|_{\mathrm{L}^2(0,T)}\big)\bigg)\no\\&\quad +\|\boldsymbol{\omega}\|_{\H^2}\bigg(2 \|\u_0\|_{\H}+T^{1/2}\sup_{t\in[0,T]}\|\g(t)\|_{\L^2}\big(\|f_1\|_{\mathrm{L}^2(0,T)}+\|f_2\|_{\mathrm{L}^2(0,T)}\big)\bigg)^{r-1}\bigg]^2.
		\end{align}
		A combination of the inequalities \eqref{3p} and \eqref{3q1} gives the following estimate:
		\begin{align}\label{3r}
			\left\|\mathcal{A}f_1-\mathcal{A}f_2\right\|_{\mathrm{L}^2(0,T)} \leq (m_6T)^{1/2} \|f_1-f_2\|_{\mathrm{L}^2(0,T)},
		\end{align}
		where
		\begin{align*}
			m_6&= \frac{C}{g_0^2}\sup_{t\in[0,T]}\|\g(t)\|_{\L^2}^2\exp \bigg\{\|\u_0\|_{\H}^2+\widetilde{a}^2\sup_{t\in[0,T]}\|\g(t)\|_{\L^2}^2\bigg\}\no\\&\quad\times\bigg[ \|\Delta \boldsymbol{\omega}\|_{\H}+\|\boldsymbol{\omega}\|_{\H}+\|\nabla \boldsymbol{\omega}\|_{\L^{\infty}} \bigg( \|\u_0\|_{\H}+T^{1/2}\widetilde{a}\sup_{t\in[0,T]}\|\g(t)\|_{\L^2}\bigg) \\&\quad+ \|\boldsymbol{\omega}\|_{\H^2}\bigg( \|\u_0\|_{\H}+T^{1/2}\widetilde{a}\sup_{t\in[0,T]}\|\g(t)\|_{\L^2}\bigg)^{r-1}\bigg]^2\ \text{ and }\\ 
			\widetilde{a}&=a+\|\varphi'/g_1\|_{\mathrm{L}^2(0,T)},
		\end{align*}
		and $a$ is the radius of the ball $\mathcal{D}$.
		The quantity $m_6$ is also expressed in terms of the  input data only.\
		
		Using the same arguments  those made in the case for $r\in(2,3]$, we can define the operator $\mathcal{A}^k$. Using the standard technique of mathematical induction on $k,$ the  inequality \eqref{3r} validates the estimate
		\begin{align}\label{3s}
			\left\|\mathcal{A}^kf_1-\mathcal{A}^kf_2\right\|_{\mathrm{L}^2(0,T)} \leq \bigg(\frac{m_6^kT^k}{k!}\bigg)^{1/2} \|f_1-f_2\|_{\mathrm{L}^2(0,T)}.
		\end{align}
		It is clear that
		\begin{align*}
			\big(m_6^kT^k/k!\big)^{1/2} \rightarrow 0  \ \ \text{ as } \ \ k \rightarrow \infty.
		\end{align*}
		Therefore, there exists a positive integer $\tilde k_0$ such that
		\begin{align*}
			\big(m_6^{\tilde k_0}T^{\tilde k_0} /\tilde k_{0}!\big)^{1/2} <1.
		\end{align*}
		From the estimate \eqref{3s}, it is immediate that the operator $\mathcal{A}^{\tilde k_0}$ is a contracting mapping in the ball $\mathcal{D}$. 
	\end{proof}
	Let us now prove  the existence of a solution $(\u,p,f)$ to the inverse problem \eqref{1a}-\eqref{1e}. 
	\begin{proof}[Proof of Theorem \ref{thm2} (i)-Existence]
Since	the operator $\mathcal{A}^k$ is a contraction mapping in the ball $\mathcal{D}$, an application of the extension of the contraction mapping theorem (Theorem \ref{thm3.1}) yields the operator $\mathcal{A}$ has a unique fixed point. Hence, the inverse problem \eqref{1a}-\eqref{1e} has a solution, which completes the proof. 
	\end{proof}
	\subsection{Uniqueness} Let us now prove the uniqueness of the solution obtained in Theorem \ref{thm2} (i). 
	\begin{proof}[Proof of Theorem \ref{thm2} (i)-Uniqueness]
	Assume on the contrary that there are two distinct solutions, $(\u_1,p_1,f_1)$ and  $(\u_2,p_2,f_2)$ of the inverse problem \eqref{1a}-\eqref{1e} such that both $f_1$ and $f_2$ lie within the ball $\mathcal{D}$.
	
	We wish to prove that $f_1$ does not coincide with $f_2$ a.e. on $[0,T]$. Indeed, if $f_1$ is equal to $f_2$ a.e. on $[0,T]$, then $(\u_1,p_1)$ will coincide with $(\u_2,p_2)$ a.e. in $\Omega \times [0,T],$ which is in accordance with the uniqueness theorem for the direct problem.
	
	We begin by analyzing the first triplet $(\u_1,p_1,f_1)$. Using similar arguments as in the proof of Theorem \ref{thm1}, we can conclude that the function $f_1$ represents a solution to the equation \eqref{1h}.  Similar arguments also prove that the function $f_2$ also solves the same equation \eqref{1h}. But we have just proved the uniqueness of the solution to the equation \eqref{1h} in $\mathcal{D}$. Thus, we have arrived at the conclusion that the assumption about the existence of two distinct solutions $(\u_i,p_i,f_i) \ (i=1,2)$ does not hold. Hence, the inverse problem \eqref{1a}-\eqref{1e} has a unique solution, which completes the proof.
	\end{proof} 
We will now give an example (motivated from \cite{IAV}) to show that the class of functions satisfying the conditions of part (i) of the Theorem \ref{thm2} is not empty.
\begin{example}
		Let $\Omega \subset \mathbb{R}^2, \ \u_0 \in \H, \ \g=\g(x), \ \g \in \L^2(\Omega), \ \boldsymbol{\omega} \in \H^2(\Omega) \cap \V, \ \nabla \boldsymbol{\omega} \in \widetilde{\L}^{\infty}, \ \varphi \in \mathrm{H}^1(0,T)$ and $$\left| \int_\Omega\g(x) \cdot \boldsymbol{\omega}(x) \d x\right| = g_0 >0.$$
		If we consider
		\begin{align*}
			\int_\Omega\g(x) \cdot \boldsymbol{\omega}(x) \d x\equiv \varphi(t),
		\end{align*}
	that is, $\varphi \equiv \text{constant}$, and if we set $a=1$, it is clear to observe that $\widetilde{a}=1$, since $\widetilde{a}=a+\|\varphi'/g_1\|_{\mathrm{L}^2(0,T)}$. 
	\vskip 0.2 cm
	\noindent\textbf{Case I:} \emph{ $r\in (2,3]$.} The inequality \eqref{3a} has the following form:
		\begin{align}\label{ex}
		&\frac{C}{g_0} \bigg[T\bigg\{
		\big(\|\Delta \boldsymbol{\omega}\|_{\H}+\|\boldsymbol{\omega}\|_{\H}\big)\bigg(\|\u_0\|_{\H}+T^{1/2}\sup_{t\in[0,T]}\|\g(t)\|_{\L^2}\bigg)\nonumber\\&\qquad+\|\nabla \boldsymbol{\omega}\|_{\widetilde{\L}^{\infty}}\bigg(\|\u_0\|_{\H}^2+\sup_{t\in[0,T]}\|\g(t)\|_{\L^2}^2\bigg)\bigg\}^2\nonumber\\&\qquad+ T^\frac{3-r}{r-1}\|\boldsymbol{\omega}\|_{\H^2}^2\bigg(\|\u_0\|_{\H}+T^{1/2}\sup_{t\in[0,T]}\|\g(t)\|_{\L^2}\bigg)^\frac{4}{(r-1)}
		\no\\&\qquad \times \bigg(\|\u_0\|_{\H}^2+\sup_{t\in[0,T]}\|\g(t)\|_{\L^2}^2\bigg)^\frac{2(r-2)}{(r-1)}\bigg]^\frac{1}{2}<1.
	\end{align}
Clearly, as $T \to 0^+$, the left-hand side of \eqref{ex} approaches zero for $r\in(2,3)$. As a result, there exists a time $T'>0$ such that for any $T \in (0,T'],$ the  estimate \eqref{ex} holds true. Based on this reasoning, it appears reasonable to investigate the inverse problem \eqref{1a}-\eqref{1e}, with $T \in (0, T'] $ and $a=1$, the radius of the ball $\mathcal{D}$. 

For $r=3$, the condition becomes $\frac{C}{g_0}\|\boldsymbol{\omega}\|_{\H^2}\|\u_0\|_{\H}\left(\|\u_0\|_{\H}^2+\|\g(0)\|_{\L^2}^2\right)^{\frac{1}{2}}<1$. Under this smallness assumption on the data, one can obtain the existence and uniqueness of the inverse problem \eqref{1a}-\eqref{1e} for any $T \in (0,T']$. 
\vskip 0.2 cm
\noindent\textbf{Case II:} \emph{ $r\in [1,2]$.} The inequality \eqref{3b} has the following form:
\begin{align}\label{ex1}
	&\frac{CT^{1/2}}{g_0} \bigg\{
	\big(\|\Delta \boldsymbol{\omega}\|_{\H}+\|\boldsymbol{\omega}\|_{\H}\big)\bigg(\|\u_0\|_{\H}+T^{1/2}\sup_{t\in[0,T]}\|\g(t)\|_{\L^2}\bigg)\no\\&\quad+\|\nabla \boldsymbol{\omega}\|_{\widetilde{\L}^{\infty}}\bigg(\|\u_0\|_{\H}^2+ \sup_{t\in[0,T]}\|\g(t)\|_{\L^2}^2\bigg)
	\no\\&\quad +\boldsymbol{\omega}\|_{\H^2}\bigg(\|\u_0\|_{\H}+T^{1/2}\sup_{t\in[0,T]}\|\g(t)\|_{\L^2}\bigg)^r\bigg\}<1.
\end{align}
Clearly, as $T \to 0^+$, the left-hand side of \eqref{ex1} approaches zero. Thus, there exists a time $T'>0$ such that for any $T \in (0,T'],$ the estimate \eqref{ex1} holds true. Based on this reasoning, it seems acceptable to investigate the inverse problem \eqref{1a}-\eqref{1e}, with $T \in (0, T'] $ and $a=1$, the radius of the ball $\mathcal{D}$.

Hence, for $r \in [1,3]$, it is apparent that the inverse problem with these input data satisfies the conditions of part (i) of Theorem \ref{thm2}.
\end{example}
	\subsection{Stability}
	In order to get a result on the stability, we first provide some supporting Lemmas. Let $(\u_i,p_i,f_i)$ $(i=1,2)$ be the solutions of inverse problem \eqref{1a}-\eqref{1e} corresponding to the given data $(\u_{0i},\varphi_i,\g_i)$ $(i=1,2)$ and set 
	\begin{align*}
		\u&:=\u_1-\u_2, \ \ \ \ \  \ p:=p_1-p_2, \  \ \ \ \ f:=f_1-f_2,\\ \u_0&:=\u_{01}-\u_{02},  \ \ \ \varphi:=\varphi_1-\varphi_2, \ \ \  \ \  \g:=\g_1-\g_2.
	\end{align*}
	The following lemma establishes the stability of the velocity field  $\u(\cdot)$ of the solution of the inverse problem.
	\begin{lemma}\label{lemma11}
		Let $\u_{0i} \in \H, \ \g_i \in \C( [0,T];\L^2(\Omega))$ and $f_i \in \mathrm{L}^2(0,T)$, for $i=1,2$. Then, the following estimate holds:
		\begin{align}\label{S1}
			&\sup_{t\in[0,T]}\|\u(t)\|_{\H}^2+\mu \int_0^T\|\u(t)\|_{\V}^2\d t +\frac{\beta}{2^{r-2}}\int_0^T\|\u(t)\|_{\widetilde{\L}^{r+1}}^{r+1}\d t \nonumber\\&\leq C \bigg(\|\u_0\|_{\H}^2+ \|f\|_{\mathrm{L}^2(0,T)}^2+\sup_{t\in[0,T]}\|\g(t)\|_{\L^2}^2\bigg),
		\end{align}
	where $C$ depends on the input data, $\mu,\alpha,\beta,r,T$ and $\Omega$. 
	\end{lemma}
	\begin{proof}
		Subtracting the equations \eqref{1a} for $\{\u_i,p_i, f_i\} \ (i=1,2)$, we deduce that
		\begin{align}\label{S2}
			&\u_t(t)-\mu \Delta \u(t)+(\u_1(t) \cdot \nabla)\u(t)\nonumber+(\u (t)\cdot \nabla)\u_2(t)+\alpha \u(t)\\&\quad+\beta \left(|\u_1(t)|^{r-1}\u_1(t)-|\u_2(t)|^{r-1}\u_2(t)\right)+\nabla p(t)=f (t)\g_1(t)+f_2 (t)\g(t),
		\end{align}
	for a.e. $t\in[0,T]$ in $\V'$. Taking the inner product with $\u(\cdot)$ in the equation \eqref{S2}, we obtain
		\begin{align}\label{S3}
			&\frac{1}{2}\frac{\d}{\d t}\|\u(t)\|_{\H}^2+\mu \|\u(t)\|_{\V}^2+\alpha\|\u(t)\|_{\H}^2\nonumber\\&=(f(t) \g_1(t)+f_2(t) \g(t),\u(t)) -\langle(\u(t) \cdot \nabla)\u_2(t),\u(t)\rangle\nonumber\\&\quad-\beta \langle|\u_1(t)|^{r-1}\u_1(t)-|\u_2(t)|^{r-1}\u_2(t),\u(t)\rangle.
		\end{align}
		Using H\"older's and Young's inequalities, we estimate $|(f \g_1+f_2 \g,\u)|$ as
		\begin{align}\label{S4}
			|(f \g_1+f_2 \g,\u)|&\leq \left(|f|\|\g_1\|_{\L^2} +|f_2| \|\g\|_{\L^2}\right)\|\u\|_{\H}\nonumber\\&\leq\frac{1}{ \alpha}\left(|f|^2\|\g_1\|_{\L^2}^2+|f_2|^2\|\g\|_{\L^2}^2\right)+\frac{\alpha}{2}\|\u\|_{\H}^2.
		\end{align}
		Making use of H\"older’s, Ladyzhenskaya's, and Young’s inequalities, we have
		\begin{align}\label{S5}
			|((\u \cdot \nabla)\u_2,\u)|&\leq\|\u\|_{\widetilde{\L}^{4}}^2\|\nabla\u_2\|_{\H}\no\\&\leq\sqrt{2}\|\u\|_{\H}\|\u\|_{\V}\|\u_2\|_{\V}\leq\frac{\mu}{2}\|\u\|_{\V}^2+\frac{1}{\mu}\|\u\|_{\H}^2\|\u_2\|_{\V}^2.
		\end{align}
		An estimate similar to \eqref{E11} gives
		\begin{align}\label{S6}
			\beta(|\u_1|^{r-1}\u_1-|\u_2|^{r-1}\u_2,\u)& \geq \frac{\beta}{2}\big\||\u_1|^{\frac{r-1}{2}}\u\big\|_{\H}^2+\frac{\beta}{2}\big\||\u_2|^{\frac{r-1}{2}}\u\big\|_{\H}^2 \no\\&\geq \frac{\beta}{2^{r-1}}\|\u\|_{\widetilde{\L}^{r+1}}^{r+1}.
		\end{align}
		Substituting \eqref{S4}-\eqref{S6} in \eqref{S3}, and then integrating from $0$ to $t$, we obtain
		\begin{align}\label{S7}
			&\|\u(t)\|_{\H}^2+\mu \int_0^t\|\u(s)\|_{\V}^2\d  s+\alpha\int_0^t\|\u(s)\|_{\H}^2\d s +\frac{\beta}{2^{r-2}}\int_0^t\|\u(s)\|_{\widetilde{\L}^{r+1}}^{r+1}\d s \nonumber\\&\leq \|\u_0\|_{\H}^2+\frac{2}{\mu}\int_0^t\|\u_2(s)\|_{\V}^2 \|\u(s)\|_{\H}^2 \d s\nonumber \\&\quad + \frac{2}{\alpha}\int_0^t\left(|f(s)|^2\|\g_1(s)\|_{\L^2}^2+|f_2(s)|^2\|\g(s)\|_{\L^2}^2\right)\d s,
		\end{align} 
		for all $t \in [0,T]$.	Applying Gronwall's inequality in \eqref{S7}, it can be seen that
		\begin{align*}
			\|\u(t)\|_{\H}^2&\leq \bigg\{\|\u_0\|_{\H}^2+ \frac{2}{\alpha}\bigg(\sup_{t\in[0,T]}\|\g_1(t)\|_{\L^2}^2\|f\|_{\mathrm{L}^2(0,T)}^2+\sup_{t\in[0,T]}\|\g(t)\|_{\L^2}^2\|f_2\|_{\mathrm{L}^2(0,T)}^2\bigg)\bigg\}\\&\quad \times\exp \left(\frac{2}{\mu}\int_0^T\|\u_2(t)\|_{\V}^2\d t\right),
		\end{align*}
		for all $t \in [0,T]$. Thus, from \eqref{S7}, it is immediate that
		\begin{align*}
			&\sup_{t\in[0,T]}\|\u(t)\|_{\H}^2+\mu \int_0^T\|\u(t)\|_{\V}^2\d t +\frac{\beta}{2^{r-2}}\int_0^T\|\u(t)\|_{\widetilde{\L}^{r+1}}^{r+1}\d t \\& \leq \bigg\{\|\u_0\|_{\H}^2+ \frac{2}{\alpha}\bigg(\sup_{t\in[0,T]}\|\g_1(t)\|_{\L^2}^2\|f\|_{\mathrm{L}^2(0,T)}^2+\sup_{t\in[0,T]}\|\g(t)\|_{\L^2}^2\|f_2\|_{\mathrm{L}^2(0,T)}^2\bigg)\bigg\}\\&\quad \times \exp \left(\frac{2}{\mu}\int_0^T\|\u_2(t)\|_{\V}^2\d t\right) \\&\leq\bigg\{\|\u_0\|_{\H}^2+ \frac{2}{\alpha}\bigg(\sup_{t\in[0,T]}\|\g_1(t)\|_{\L^2}^2\|f\|_{\mathrm{L}^2(0,T)}^2+\sup_{t\in[0,T]}\|\g(t)\|_{\L^2}^2\|f_2\|_{\mathrm{L}^2(0,T)}^2\bigg)\bigg\}\\& \quad\times 
			\exp\bigg\{\frac{1}{\mu^2}\bigg(\|\u_{02}\|_{\H}^2+\frac{1}{\alpha}\sup_{t\in[0,T]}\|\g_2(t)\|_{\L^2}^2\|f_2\|_{\mathrm{L}^2(0,T)}^2\bigg)\bigg\}\\&
				\leq C\big(\mu,\alpha, \beta,r,\|\u_{02}\|_{\H},\|f_2\|_{\L^2},\|\g_2\|_{0},T\big)\bigg(\|\u_0\|_{\H}^2+ \|f\|_{\mathrm{L}^2(0,T)}^2+\sup_{t\in[0,T]}\|\g(t)\|_{\L^2}^2\bigg),
		\end{align*} 
		and \eqref{S1} follows.
	\end{proof}
	The next lemma establishes the stability of the scalar function $f$ of the solution of the inverse problem.
	\begin{lemma}\label{lemma2}
		Let $\u_{0i} \in \H, \ \boldsymbol{\omega} \in \H^2(\Omega) \cap \V, \ \nabla \boldsymbol{\omega} \in \widetilde{\L}^\infty$, $\g_i \in \C( [0,T];\L^2(\Omega)),\ \varphi_i\in \mathrm{H}^1(0,T)$, and $f_i \in \mathrm{L}^2(0,T)$, for $i=1,2$. Then, the following estimate holds:
		\begin{align}\label{S8}
			\|f\|_{\mathrm{L}^2(0,T)}\leq C\bigg(\|\u_0\|_{\H}+\|\varphi\|_{\mathrm{H}^1(0,T)}+\sup_{t\in[0,T]}\|\g(t)\|_{\L^2}\bigg),
		\end{align}
		where $C$ depends on the input data, $\mu,\alpha,\beta,r,T$ and $\Omega$.
	\end{lemma}
	\begin{proof}
		Taking the inner product with $\boldsymbol{\omega}(x)$ in the equation \eqref{S2}, we obtain
		\begin{align}\label{S9}
			&\int_{\Omega}\u_t\cdot \boldsymbol{\omega} \d x-\int_{\Omega}\mu \Delta \u \cdot \boldsymbol{\omega} \d x+\int_{\Omega}(\u_1 \cdot \nabla)\u \cdot \boldsymbol{\omega} \d x\nonumber+\int_{\Omega}(\u \cdot \nabla)\u_2 \cdot \boldsymbol{\omega} \d x\no\\&\quad+\alpha\int_{\Omega} \u \cdot \boldsymbol{\omega} \d x+\beta \int_{\Omega} \left(|\u_1|^{r-1}\u_1-|\u_2|^{r-1}\u_2\right)\cdot \boldsymbol{\omega} \d x\no\\&\quad=\int_{\Omega}f \g_1 \cdot \boldsymbol{\omega} \d x+\int_{\Omega}f_2 \g \cdot \boldsymbol{\omega} \d x.
		\end{align}
		From the above equation, we have the following estimate:
		\begin{align}\label{S91}
			|f(t)|^2g_0^2 \no&\leq 2|f_2(t)|^2\|\g(t)\|_{\L^2}^2\|\boldsymbol{\omega}\|_{\H}^2 +C\bigg[|\varphi'(t)|^2+\|\u(t)\|_{\H}^2\|\Delta \boldsymbol{\omega}\|_{\H}^2\no\\&\quad+\|\nabla \boldsymbol{\omega}\|_{\widetilde{\L}^{\infty}}^2\|\u(t)\|_{\H}^2\big(\|\u_1(t)\|_{\H}^2+\|\u_2(t)\|_{\H}^2\big)+ |\varphi(t)|^2\bigg]\no\\&\quad+C\|\boldsymbol{\omega}\|_{\widetilde{\L}^\infty}^2\left\||\u_1(t)|^{r-1}\u_1(t)-|\u_2(t)|^{r-1}\u_2(t)\right\|_{\L^1}^2.
		\end{align}
		Integrating the estimate \eqref{S91} from $0$ to $T$, we deduce that
		\begin{align}\label{S92}
			&g_0^2\|f\|_{\mathrm{L}^2(0,T)}^2 
			\no\\&\leq 2 \sup_{t\in[0,T]}\|\g(t)\|_{\L^2}^2\|\boldsymbol{\omega}\|_{\H}^2 \|f_2\|_{\mathrm{L}^2(0,T)}^2+C\bigg[\|\varphi\|_{\mathrm{L}^2(0,T)}^2+\|\varphi'\|_{\mathrm{L}^2(0,T)}^2\no\\&\quad+T \sup_{t\in[0,T]}\|\u(t)\|_{\H}^2\bigg\{\|\Delta \boldsymbol{\omega}\|_{\H}^2+\|\nabla \boldsymbol{\omega}\|_{\widetilde{\L}^{\infty}}^2\bigg(\sup_{t\in[0,T]}\|\u_1(t)\|_{\H}^2+\sup_{t\in[0,T]}\|\u_2(t)\|_{\H}^2\bigg)\bigg\}\bigg]
			\no\\&\quad+C\|\boldsymbol{\omega}\|_{\H^2}^2\int_0^T\left\||\u_1(t)|^{r-1}\u_1(t)-|\u_2(t)|^{r-1}\u_2(t)\right\|_{\L^1}^2 \d t.
		\end{align}
		\vskip 0.2 cm
		\noindent\textbf{Case I:} \emph{ $r\in (2,3]$.}
		Using the estimate \eqref{3j1} in \eqref{S92}, one can deduce that
		\begin{align}\label{S10}
			&g_0^2\|f\|_{\mathrm{L}^2(0,T)}^2 
			\no\\&\leq 2 \sup_{t\in[0,T]}\|\g(t)\|_{\L^2}^2\|\boldsymbol{\omega}\|_{\H}^2 \|f_2\|_{\mathrm{L}^2(0,T)}^2+C\bigg[\|\varphi\|_{\mathrm{L}^2(0,T)}^2+\|\varphi'\|_{\mathrm{L}^2(0,T)}^2\no\\&\quad+T \sup_{t\in[0,T]}\|\u(t)\|_{\H}^2\bigg\{\|\Delta \boldsymbol{\omega}\|_{\H}^2+\|\nabla \boldsymbol{\omega}\|_{\widetilde{\L}^{\infty}}^2\bigg(\sup_{t\in[0,T]}\|\u_1(t)\|_{\H}^2+\sup_{t\in[0,T]}\|\u_2(t)\|_{\H}^2\bigg)\bigg\}\bigg]\no\\&\quad +
			CT^{3-r}\|\boldsymbol{\omega}\|_{\H^2}^2 \sup_{t\in[0,T]}\|\u(t)\|_{\H}^2\bigg\{\sup_{t\in[0,T]} \| \u_1(t)\|_{\H}^{2}\bigg(\int_0^T \| \u_1(t)\|_{\V}^2 \d t\bigg)^{r-2}\no\\&\quad+\sup_{t\in[0,T]}\| \u_2(t)\|_{\H}^{2}\bigg(\int_0^T \| \u_2(t)\|_{\V}^2 \d t\bigg)^{r-2}\bigg\}\bigg].
		\end{align}
		Substituting the estimates \eqref{E1}, \eqref{E2}, and \eqref{S1} in \eqref{S10}, we arrive at
		\begin{align*}
			\|f\|_{\mathrm{L}^2(0,T)}^2 \leq C\bigg(\|\u_0\|_{\H}^2+\|\varphi\|_{\mathrm{L}^2(0,T)}^2+\|\varphi'\|_{\mathrm{L}^2(0,T)}^2+\sup_{t\in[0,T]}\|\g(t)\|_{\L^2}^2\bigg),
		\end{align*}
		and finally we get
		\begin{align*}
			\|f\|_{\mathrm{L}^2(0,T)}\leq C\bigg(\|\u_0\|_{\H}+\|\varphi\|_{\mathrm{H}^1(0,T)}+\sup_{t\in[0,T]}\|\g(t)\|_{\L^2}\bigg),
		\end{align*}
		which is the estimate \eqref{S8}.
		\vskip 0.2 cm
		\noindent\textbf{Case II:} \emph{ $r\in [1,2]$.}  Substituting the estimate \eqref{3o1} in \eqref{S92}, we obtain 
		\begin{align}\label{S11}
			&g_0^2\|f\|_{\mathrm{L}^2(0,T)}^2 
			\no\\&\leq 2 \sup_{t\in[0,T]}\|\g(t)\|_{\L^2}^2\|\boldsymbol{\omega}\|_{\H}^2 \|f_2\|_{\mathrm{L}^2(0,T)}^2+C\bigg[\|\varphi\|_{\mathrm{L}^2(0,T)}^2+\|\varphi'\|_{\mathrm{L}^2(0,T)}^2\no\\&\quad+T \sup_{t\in[0,T]}\|\u(t)\|_{\H}^2\bigg\{\|\Delta \boldsymbol{\omega}\|_{\H}^2+\|\nabla \boldsymbol{\omega}\|_{\widetilde{\L}^{\infty}}^2\bigg(\sup_{t\in[0,T]}\|\u_1(t)\|_{\H}^2+\sup_{t\in[0,T]}\|\u_2(t)\|_{\H}^2\bigg)\bigg\}\bigg]\no\\&\quad +CT\|\boldsymbol{\omega}\|_{\H^2}^2\sup_{t\in[0,T]}\|\u(t)\|_{\H}^2
			\bigg(\sup_{t\in[0,T]} \| \u_1(t)\|_{\H}^{2(r-1)}+\sup_{t\in[0,T]}\| \u_2(t)\|_{\H}^{2(r-1)}\bigg).
		\end{align}
		Using the estimates \eqref{E1}, \eqref{E2}, and \eqref{S1} in \eqref{S11}, one can easily deduce the estimate \eqref{S8}, which completes the proof.
	\end{proof}
\begin{proof}[Proof of part (ii) of Theorem \ref{thm2}]
		\emph{Stability of the pressure field $ p$:} The equation \eqref{S2} can be used directly to demonstrate the stability of the pressure field $p$.
	Taking divergence on both sides of the equation \eqref{S2}, we deduce that
	\begin{align*}
		-\Delta p &=-\nabla \cdot(f\g_1+f_2\g)+\big\{\nabla \cdot \big[(\u_1 \cdot \nabla)\u+(\u \cdot \nabla)\u_2+\beta \left(|\u_1|^{r-1}\u_1-|\u_2|^{r-1}\u_2\right)\big] \big\},
	\end{align*}
in the weak sense. The above equation yields
\begin{align}\label{S12}
		 p&=(-\Delta)^{-1}\big\{\nabla \cdot \big[-(f\g_1+f_2\g)+\nabla \cdot(\u_1 \otimes \u)\no\\&\quad+\nabla \cdot(\u \otimes \u_2)+\beta \left(|\u_1|^{r-1}\u_1-|\u_2|^{r-1}\u_2\right)\big] \big\}.
	\end{align}
	 Taking $\mathrm{L}^2$-norm on both  sides in the equation \eqref{S12}, and then using elliptic regularity (Cattabriga's regularity theorem), H\"older's inequality and Taylor's formula, we have
	\begin{align}\label{S13}
		\| p \|_{\mathrm{L}^2} &=\big\|(-\Delta)^{-1}\big\{\nabla \cdot \big[-(f\g_1+f_2\g)+\nabla \cdot(\u_1 \otimes \u)\no\\&\qquad+\nabla \cdot(\u \otimes \u_2)+\beta \left(|\u_1|^{r-1}\u_1-|\u_2|^{r-1}\u_2\right)\big] \big\}\big\|_{\mathrm{L}^2}\no\\&\leq C
		\big\|\big\{\nabla \cdot \big[-(f\g_1+f_2\g)+\nabla \cdot(\u_1 \otimes \u)\no\\&\qquad+\nabla \cdot(\u \otimes \u_2)+\beta \left(|\u_1|^{r-1}\u_1-|\u_2|^{r-1}\u_2\right)\big] \big\}\big\|_{\H^{-2}}
		\no\\&\leq C\|f\g_1+f_2\g\|_{\V'}+ C\| \u_1 \otimes \u\|_{\H}+C\| \u\otimes \u_2\|_{\H}\no\\&\qquad+C\beta\left\| |\u_1|^{r-1}\u_1-|\u_2|^{r-1}\u_2\right\|_{\V'} \no\\
		& \leq C\|f\g_1+f_2\g\|_{\L^2}+C \left(\| \u_1\|_{\widetilde{\L}^4}+\| \u_2\|_{\widetilde{\L}^4}\right) \|\u\|_{\widetilde{\L}^4}\no\\&\qquad+C\left\| |\u_1|^{r-1}\u_1-|\u_2|^{r-1}\u_2\right\|_{\widetilde{\L}^\frac{r+1}{r}}\no\\
		& \leq C |f|\|\g_1\|_{\L^2}+C|f_2|\|\g\|_{\L^2}+C\left(\| \u_1\|_{\widetilde{\L}^4}+\| \u_2\|_{\widetilde{\L}^4}\right) \|\u\|_{\widetilde{\L}^4}\no\\&\quad+C\|\u_1-\u_2\|_{\widetilde{\L}^{r+1}}\left(\|\u_1\|_{\widetilde{\L}^{r+1}}^{r-1}+\|\u_2\|_{\widetilde{\L}^{r+1}}^{r-1}\right),
	\end{align}
	where we have used the embedding $\V \subset \widetilde{\L}^{r+1} \subset \H \equiv \H'\subset \widetilde{\L}^\frac{r+1}{r} \subset \V'$. Taking the $(\frac{r+1}{r})^\text{th}$ power on both sides in \eqref{S13} and then integrating the resulting inequality from $0$ to $T$, followed by applying H\"older's inequality, we deduce that
	\begin{align*}
		\int_0^T\| p(t) \|_{\mathrm{L}^2}^\frac{r+1}{r} \d t &\leq C \sup_{t\in[0,T]}\|\g_1(t)\|_{\L^2}^\frac{r+1}{r}\int_0^T|f(t)|^\frac{r+1}{r}\d t+ C \sup_{t\in[0,T]}\|\g(t)\|_{\L^2}^\frac{r+1}{r}\int_0^T|f_2(t)|^\frac{r+1}{r}\d t\\&\quad+C\int_0^T\| \u_1(t)\|_{\widetilde{\L}^4}^\frac{r+1}{r}\|\u(t)\|_{\widetilde{\L}^4}^\frac{r+1}{r}\d t+C \int_0^T\| \u_2(t)\|_{\widetilde{\L}^4}^\frac{r+1}{r}\|\u(t)\|_{\widetilde{\L}^4}^\frac{r+1}{r}\d t\\&\quad+C\int_0^T\|\u_1(t)\|_{\widetilde{\L}^{r+1}}^\frac{(r-1)(r+1)}{r} \|\u(t)\|_{\widetilde{\L}^{r+1}}^\frac{r+1}{r}\d t\\&\quad+C\int_0^T\|\u_2(t)\|_{\widetilde{\L}^{r+1}}^\frac{(r-1)(r+1)}{r} \|\u(t)\|_{\widetilde{\L}^{r+1}}^\frac{r+1}{r}\d t
			\\&	\leq CT^\frac{r-1}{2r} \sup_{t\in[0,T]}\|\g_1(t)\|_{\L^2}^\frac{r+1}{r}\|f\|_{\mathrm{L}^2(0,T)}^\frac{r+1}{r}+ CT^\frac{r-1}{2r} \sup_{t\in[0,T]}\|\g(t)\|_{\L^2}^\frac{r+1}{r}\|f_2\|_{\mathrm{L}^2(0,T)}^\frac{r+1}{r}\\&\quad+CT^\frac{r-1}{2r}\bigg(\int_0^T\| \u(t)\|_{\widetilde{\L}^4}^4\d t\bigg)^\frac{r+1}{4r}\bigg(\int_0^T\| \u_1(t)\|_{\widetilde{\L}^4}^{4}\d t\bigg)^\frac{r+1}{4r}
			\\&\quad+CT^\frac{r-1}{2r}\bigg(\int_0^T\| \u(t)\|_{\widetilde{\L}^4}^4\d t\bigg)^\frac{r+1}{4r}\bigg(\int_0^T\| \u_2(t)\|_{\widetilde{\L}^4}^4\d t\bigg)^\frac{r+1}{4r}
			\\&\quad+C\bigg(\int_0^T\| \u(t)\|_{\widetilde{\L}^{r+1}}^{r+1}\d t\bigg)^\frac{1}{r}\bigg(\int_0^T\| \u_1(t)\|_{\widetilde{\L}^{r+1}}^{r+1}\d t\bigg)^\frac{r-1}{r}
			\\&\quad+C\bigg(\int_0^T\| \u(t)\|_{\widetilde{\L}^{r+1}}^{r+1}\d t\bigg)^\frac{1}{r}\bigg(\int_0^T\| \u_2(t)\|_{\widetilde{\L}^{r+1}}^{r+1}\d t\bigg)^\frac{r-1}{r}.
	\end{align*}
Substituting the energy estimates \eqref{E2}, \eqref{S1} and \eqref{S8} in the above estimate, we arrive at
\begin{align*}
	\|  p \|_{\mathrm{L}^\frac{r+1}{r}(0,T;\mathrm{L}^2(\Omega))}
	& \leq C\bigg(\|\u_0\|_{\H}+ \sup_{t\in[0,T]}\|\g(t)\|_{\L^2}+\|f\|_{\L^2}\bigg)
	\\& \leq C\bigg(\|\u_0\|_{\H}+\|\varphi\|_{\mathrm{H}^1(0,T)}+ \sup_{t\in[0,T]}\|\g(t)\|_{\L^2}\bigg),
\end{align*}
	where $C$ depends on the input data, $\mu,\alpha,\beta,r$,$T$ and $\Omega$.
	We can see that the solution depends continuously on the data from the stability estimate of the pressure field $p$ and Lemmas \ref{lemma11}-\ref{lemma2}. We infer the Lipschitz stability of the solution $(\u,  p, f)$ as
	\begin{align*}
		\nonumber&\|\u_1-\u_2\|_{\mathrm{L}^\infty(0,T;\H)}+	\|\u_1-\u_2\|_{\mathrm{L}^2(0,T;\V)}+	\|\u_1-\u_2\|_{\mathrm{L}^{r+1}(0,T;\widetilde{\L}^{r+1})}\\&\quad+\|p_1-p_2 \|_{\mathrm{L}^\frac{r+1}{r}(0,T;\mathrm{L}^2(\Omega))}+\|f_1-f_2\|_{\mathrm{L}^2(0,T)} \nonumber\\&\leq C\bigg(\|\u_{01}-\u_{02}\|_{ \H}+\sup_{t\in[0,T]}\|(\g_1-\g_2)(t)\|_{\L^2}+\|\varphi_1-\varphi_2\|_{\mathrm{H}^1(0,T)}\bigg),
	\end{align*}
which completes the proof of part (ii) of Theorem \ref{thm2}.
	\end{proof} 
		\begin{appendix}
		\renewcommand{\thesection}{\Alph{section}}
		\numberwithin{equation}{section}
		\section{Energy estimates}\label{sec4}\setcounter{equation}{0} 
		Since the existence and uniqueness of weak solutions for the system \eqref{1a}-\eqref{1d} are known, we derive  a number of a-priori estimates for the solutions. 	To obtain the energy estimates of the solutions of CBF equations \eqref{1a}-\eqref{1d}, we assume that $\u_0 \in \H $, $\g \in \C( [0,T];\L^2(\Omega))$ and $f \in \mathrm{L^2}(0,T)$.	The next lemma provides the usual energy estimates for the CBF equations \eqref{1a}-\eqref{1d}.
		\begin{lemma}\label{lemma1}
			Let $(\u(\cdot),p(\cdot))$ be the unique weak solution of the CBF  equations \eqref{1a}-\eqref{1d} and $\u_0 \in \H$. Then, the following estimate holds: 
			\begin{align}\label{E1}
				\sup_{t\in[0,T]}\|\u(t)\|_{\H} \leq \|\u_0\|_{\H}+	T^{1/2}\sup_{t\in[0,T]}\|\g(t)\|_{\L^2}\|f\|_{\mathrm{L}^2(0,T)},
			\end{align}
			and
			\begin{align}\label{E2}
				&\sup_{t\in[0,T]}\|\u(t)\|_{\H}^2+2\mu \int_0^T\|\u(t)\|_{\V}^2\d t+\alpha\int_0^T\|\u(t)\|_{\H}^2\d t+2\beta\int_0^T\|\u(t)\|_{\widetilde{\L}^{r+1}}^{r+1}\d t \nonumber\\&\leq \|\u_0\|_{\H}^2+\frac{1}{\alpha}\sup_{t\in[0,T]}\|\g(t)\|_{\L^2}^2\|f\|_{\mathrm{L}^2(0,T)}^2.
			\end{align}
		\end{lemma} 
		\begin{proof}
		Taking the inner product with $\u(\cdot)$ in the equation \eqref{1a} and use the fact that $\langle(\u \cdot \nabla)\u,\u\rangle=0$ and $\langle\nabla p,\u\rangle=0$ to obtain 
			\begin{align}\label{E3}
				\frac{1}{2}\frac{\d}{\d t}\|\u(t)\|_{\H}^2+\mu \|\u(t)\|_{\V}^2+\alpha \|\u(t)\|_{\H}^2+\beta\|\u(t)\|_{\widetilde{\L}^{r+1}}^{r+1}=(f \g(t),\u(t)),
			\end{align}
			for a.e. $t\in[0,T]$. Simplifying \eqref{E3}, we get 
			\begin{align*}
				\frac{1}{2}\frac{\d}{\d t}\|\u(t)\|_{\H}^2
				\leq |f(t)|\| \g(t)\|_{\L^2}\|\u(t)\|_{\H},
			\end{align*}
			which further reduces to
			\begin{align}\label{E4}
				\frac{\d}{\d t}\|\u(t)\|_{\H}
				\leq |f(t)|\| \g(t)\|_{\L^2}.
			\end{align}
			Integrating the inequality \eqref{E4} from $0$ to $t$, and taking supremum over $0$ to $T$ on both sides, one can easily deduce  \eqref{E1}.\
			
			Using H\"older's and Young's inequalities, we estimate $|(f \g,\u)|$ as
			\begin{align}\label{E5}
				|(f \g,\u)| \leq|f|\|\g\|_{\L^2}\|\u\|_{\H} \leq\frac{1}{2\alpha}|f|^2\|\g\|_{\L^{2}}^{2}+\frac{\alpha}{2}\|\u\|_{\H}^{2}.
			\end{align}
			Substituting \eqref{E5} in \eqref{E3}, and then integrating it from $0$ to $t$, we find 
			\begin{align*}
				&\|\u(t)\|_{\H}^2+2\mu \int_0^t\|\u(s)\|_{\V}^2\d s+\alpha\int_0^t\|\u(s)\|_{\H}^2\d s +2\beta\int_0^t\|\u(s)\|_{\widetilde{\L}^{r+1}}^{r+1}\d s\\&\quad\leq \|\u_0\|_{\H}^2+\frac{1}{\alpha}\int_0^t|f(s)|^2\|\g(s)\|_{\L^2}^2\d s,
			\end{align*}
			for all $t \in [0,T]$ and \eqref{E2} follows. 
		\end{proof}
		The next lemma establishes  stability estimates for the CBF equations \eqref{1a}-\eqref{1d}.
		\begin{lemma}
			Let $(\u_1(\cdot),p_1(\cdot))$ be the solution of the direct problem \eqref{1a}-\eqref{1d} corresponding to the external forcing $f_1 \g$ and the initial velocity $\u_{01}$; and $(\u_2(\cdot),p_2(\cdot))$ be the solution of the same problem corresponding to the external forcing $f_2 \g$ and the initial velocity $\u_{02}$. 	Then, for $\u_{01 },\u_{02}\in \H,$ $f_1,f_2\in\mathrm{L}^2(0,T)$ and $\g \in \C([0,T];\L^2(\Omega))$,  the following estimate holds:
			\begin{align}\label{E6}
				&\sup_{t\in[0,T]}\|(\u_1-\u_2)(t)\|_{\H}^2+\mu \int_0^T \|(\u_1-\u_2)(t)\|_{\V}^2\d  t\nonumber\\&\quad+\alpha\int_0^T \|(\u_1-\u_2)(t)\|_{\H}^2\d t+\frac{\beta}{2^{r-2}}\int_0^T\|(\u_1-\u_2)(t)\|_{\widetilde{\L}^{r+1}}^{r+1}\d t\nonumber\\&\leq \bigg(\|\u_{01}-\u_{02}\|_{\H}^2+\frac{1}{\alpha}\sup_{t\in[0,T]}\|\g(t)\|_{\L^2}^2\|f_1-f_2\|_{\mathrm{L}^2(0,T)}^2\bigg) \exp \left(\frac{2}{\mu}\int_0^T\|\u_2(t)\|_{\V}^2\d t \right). 
			\end{align}
			Moreover, for $\u_{01}=\u_{02}$, we have
			\begin{align}\label{E61}
				&\sup_{t\in[0,T]}\|(\u_1-\u_2)(t)\|_{\H}^2 \no\\&\leq \bigg(\frac{1}{\alpha}\sup_{t\in[0,T]}\|\g(t)\|_{\L^2}^2\|f_1-f_2\|_{\mathrm{L}^2(0,T)}^2\bigg) \exp \left(\frac{2}{\mu}\int_0^T\|\u_2(t)\|_{\V}^2\d t \right). 
			\end{align}
		\end{lemma}
		\begin{proof}
			It is clear that the pair $\big((\u_1-\u_2)(\cdot),(p_1-p_2)(\cdot)\big)$ satisfies the following system:
			\begin{equation}\label{E7}
				\left\{
				\begin{aligned}
					(\u_1-\u_2)_t-\mu \Delta (\u_1-\u_2)+\alpha(\u_1-\u_2)+\nabla(p_1-p_2)&=\boldsymbol{F}_k, \ \ \ \text{in} \ \Omega \times (0,T), \\ \nabla \cdot (\u_1-\u_2)&=0, \ \ \ \ \ \text{in} \ \Omega \times (0,T), \\
					\u_1-\u_2&=\boldsymbol{0}, \ \ \ \ \ \text{on} \ \partial\Omega \times [0,T), \\  \u_1-\u_2&=\u_{01}-\u_{02}, \ \  \text{in} \ \Omega \times \{0\},
				\end{aligned}
				\right.
			\end{equation}
			in $\V'$,	where 
			\begin{align}\label{F_k}
				\boldsymbol{F}_k=(f_1-f_2)\g-(\u_1 \cdot \nabla)\u_1+(\u_2\cdot \nabla)\u_2-\beta\big(|\u_1|^{r-1}\u_1-|\u_2|^{r-1}\u_2\big).
			\end{align}
			Taking the inner product with $(\u_1-\u_2)(\cdot)$ to the first equation in \eqref{E7}, we obtain
			\begin{align}\label{E8}
				&\frac{1}{2}\frac{\d}{\d t}\|(\u_1-\u_2)(t)\|_{\H}^2+\mu \|(\u_1-\u_2)(t)\|_{\V}^2+\alpha\|(\u_1-\u_2)(t)\|_{\H}^2\nonumber\\&=((f_1-f_2)(t)\g(t),(\u_1-\u_2)(t))-\langle((\u_1-\u_2)(t) \cdot \nabla)\u_2(t),(\u_1-\u_2)(t) \rangle\nonumber\\&\quad-\beta \langle|\u_1(t)|^{r-1}\u_1(t)-|\u_2(t)|^{r-1}\u_2(t),(\u_1-\u_2)(t) \rangle,
			\end{align} 
		for a.e. $t\in[0,T]$. 	It is worth emphasizing that, for $r\geq 1$ (see Sec. 2.4, \cite{MTM4,MTM6})
			\begin{align}\label{E9}
				&\beta\left(|\u_1|^{r-1}\u_1-|\u_2|^{r-1}\u_2,\u_1-\u_2\right)\nonumber\\&\geq \frac{\beta}{2}\big\||\u_1|^{\frac{r-1}{2}}(\u_1-\u_2)\big\|_{\H}^2+\frac{\beta}{2}\big\||\u_2|^{\frac{r-1}{2}}(\u_1-\u_2)\big\|_{\H}^2.
			\end{align}
			It can be easily seen that
			\begin{align*}
				\|\u_1-\u_2\|_{\widetilde{\L}^{r+1}}^{r+1}&=\int_{\Omega}|\u_1(x)-\u_2(x)|^{r-1} |\u_1(x)-\u_2(x)|^{2} \d x \nonumber\\&\leq 2^{r-2} \int_{\Omega}\left(|\u_1(x)|^{r-1}+|\u_2(x)|^{r-1}\right) |\u_1(x)-\u_2(x)|^{2} \d x
				\nonumber\\& \leq  2^{r-2}\left( \big\||\u_1|^{\frac{r-1}{2}}(\u_1-\u_2)\big\|_{\H}^2+\big\||\u_2|^{\frac{r-1}{2}}(\u_1-\u_2)\big\|_{\H}^2\right).
			\end{align*}
			From the above inequality, we have
			\begin{align}\label{E10}
				\frac{2^{2-r}\beta}{2}\|\u_1-\u_2\|_{\widetilde{\L}^{r+1}}^{r+1}\leq \frac{\beta}{2}\big\||\u_1|^{\frac{r-1}{2}}(\u_1-\u_2)\big\|_{\H}^2+\frac{\beta}{2}\big\||\u_2|^{\frac{r-1}{2}}(\u_1-\u_2)\big\|_{\H}^2.
			\end{align}
			Thus, from \eqref{E9} and \eqref{E10}, one can easily deduce that
			\begin{align}\label{E11}
				&\beta\left(|\u_1|^{r-1}\u_1-|\u_2|^{r-1}\u_2,\u_1-\u_2\right) \geq \frac{\beta}{2^{r-1}}\|\u_1-\u_2\|_{\widetilde{\L}^{r+1}}^{r+1}.
			\end{align}
			Using the Cauchy-Schwarz and Young's inequalities, we get
			\begin{align}\label{E12}
				\left((f_1-f_2)\g,\u_1-\u_2\right)\leq \frac{1}{2\alpha}|f_1-f_2|^2\|\g\|_{\L^2}^2+\frac{\alpha}{2}\|\u_1-\u_2\|_{\H}^2.
			\end{align}
			Making use of H\"older’s, Ladyzhenskaya's, and Young’s inequalities, we have
			\begin{align}\label{E13}
				\langle((\u_1-\u_2) \cdot \nabla)\u_2,\u_1-\u_2\rangle&\leq \|\u_1-\u_2\|_{\widetilde{\L}^4}^2\|\u_2\|_{\V}\nonumber\\& \leq \sqrt{2} \|\u_1-\u_2\|_{\H}\|\u_1-\u_2\|_{\V}\|\u_2\|_{\V} \nonumber\\&\leq \frac{\mu}{2}\|\u_1-\u_2\|_{\V}^2+\frac{1}{\mu}\|\u_1-\u_2\|_{\H}^2\|\u_2\|_{\V}^2.
			\end{align}
			Substituting \eqref{E11}-\eqref{E13} in \eqref{E8}, we obtain
			\begin{align*}
				&\frac{1}{2}\frac{\d}{\d t}\|(\u_1-\u_2)(t)\|_{\H}^2+\frac{\mu}{2} \|(\u_1-\u_2)(t)\|_{\V}^2\no\\&\quad+\frac{\alpha}{2} \|(\u_1-\u_2)(t)\|_{\H}^2+\frac{\beta}{2^{r-1}}\|(\u_1-\u_2)(t)\|_{\widetilde{\L}^{r+1}}^{r+1} \\&\leq \frac{1}{2\alpha}|(f_1-f_2)(t)|^2\|\g(t)\|_{\L^2}^2 +\frac{1}{\mu} \|(\u_1-\u_2)(t)\|_{\H}^2\|\u_2(t)\|_{\V}^2.
			\end{align*} 
			Integrating the above inequality from $0$ to $t$, we deduce 
			\begin{align}\label{E14}
				&\|(\u_1-\u_2)(t)\|_{\H}^2+\mu \int_0^t \|(\u_1-\u_2)(s)\|_{\V}^2\d s\nonumber\\&\quad+\alpha\int_0^t \|(\u_1-\u_2)(s)\|_{\H}^2\d s +\frac{\beta}{2^{r-2}}\int_0^t\|(\u_1-\u_2)(s)\|_{\widetilde{\L}^{r+1}}^{r+1}\d s  \nonumber\\&\leq \|\u_{01}-\u_{02}\|_{\H}^2+ \frac{1}{\alpha}\int_0^t|(f_1-f_2)(t)|^2\|\g(s)\|_{\L^2}^2 \d s\nonumber\\&\quad+\frac{2}{\mu} \int_0^t\|(\u_1-\u_2)(s)\|_{\H}^2 \|\u_2(s)\|_{\V}^2\d s,
			\end{align}
			for all $t \in [0,T]$. Applying Gronwall's inequality in \eqref{E14} and then taking supremum over $0$ to $T$ on both sides, we finally have \eqref{E6}. One can easily deduce \eqref{E61} by taking $\u_{01}=\u_{02}$ in \eqref{E6} and it  completes the proof. 
		\end{proof}
		\end{appendix}
	
	\medskip\noindent
	{\bf Acknowledgments:} P. Kumar and M. T. Mohan would  like to thank the Department of Science and Technology (DST), India for Innovation in Science Pursuit for Inspired Research (INSPIRE) Faculty Award (IFA17-MA110). 
	
\end{document}